\documentclass[DIV=15]{scrartcl}

\usepackage{lmodern}
\usepackage{amsmath,amssymb,amsfonts,amsthm}
\usepackage{mathtools}
\usepackage[numbers,sort&compress]{natbib}

\usepackage{algorithm}
\usepackage{algpseudocode}

\algrenewcommand\algorithmicwhile{\textbf{While}}
\algrenewcommand\algorithmicfor{\textbf{For}}
\algrenewcommand\algorithmicdo{\textbf{Do}}
\algrenewcommand\algorithmicif{\textbf{If}}
\algrenewcommand\algorithmicthen{\textbf{Then}}
\algrenewcommand\algorithmicelse{\textbf{Else}}
\algrenewcommand\algorithmicend{\textbf{End}}
\algrenewcommand\algorithmicreturn{\textbf{Return}}

\usepackage{hyperref}
\hypersetup{
  pdfauthor={Zijia Li, Hans-Peter Schröcker, Daniel F. Scharler},pdftitle={A Complete Characterization of Bounded Motion Polynomials Admitting a Factorization with Linear Factors},}

\newcommand{\D}{\mathbb{D}}
\renewcommand{\DH}{\mathbb{DH}}
\renewcommand{\H}{\mathbb{H}}

\newcommand{\R}{\mathbb{R}}
\newcommand{\N}{\mathbb{N}}
\newcommand{\adef}{\leftarrow}

\newcommand{\Clifford}[1]{C\kern-0.12em\ell_{(#1)}}

\newcommand{\evenClifford}[1]{C\kern-0.12em\ell^+_{(#1)}}

\newcommand{\qi}{\mathbf{i}}
\newcommand{\qj}{\mathbf{j}}
\newcommand{\qk}{\mathbf{k}}
\newcommand{\eps}{\varepsilon}
\newcommand{\SO}[1][3]{\mathrm{SO}(#1)}
\newcommand{\SE}[1][3]{\mathrm{SE}(#1)}

\newcommand{\Cj}[1]{{#1}^\ast}

\newcommand{\lconcat}[2]{[#1,#2]}

\newcommand{\QNorm}[1]{\nu(#1)}

\DeclareMathOperator{\mrpf}{mrpf}
\DeclareMathOperator{\lgcd}{lgcd}
\DeclareMathOperator{\rgcd}{rgcd}
\DeclareMathOperator{\mygcd}{realgcd}

\DeclareMathOperator{\lquo}{lquo}
\DeclareMathOperator{\rquo}{rquo}
\DeclareMathOperator{\lrem}{lrem}
\DeclareMathOperator{\rrem}{rrem}
\DeclareMathOperator{\czero}{czero}

\newtheorem{thm}{Theorem}
\newtheorem{lem}{Lemma}
\newtheorem{prop}{Proposition}
\newtheorem{cor}{Corollary}
\theoremstyle{definition}
\newtheorem{defn}{Definition}
\theoremstyle{remark}

\newtheorem{example}{Example}

\setkomafont{title}{\normalfont\bfseries}
\setkomafont{sectioning}{\normalfont\bfseries}

\title{Motion Polynomials Admitting a Factorization with Linear Factors}
\author{Zijia Li, Hans-Peter Schröcker, Mikhail Skopenkov, Daniel F. Scharler}

\begin{document}

\begin{center}
  \makeatletter
  {\bfseries\huge \@title}
  \par\bigskip
  \Large
  Zijia Li\footnotemark[1]\quad
  Hans-Peter Schröcker\footnotemark[2]\quad
  Mikhail Skopenkov\footnotemark[3]\quad
  Daniel F. Scharler\footnotemark[2]\textsuperscript{,}\footnotemark[4]
  \par\bigskip
  \large
  \begin{minipage}{0.7\linewidth}
    \centering
    \textsuperscript{1}KLMM, Academy of Mathematics and Systems Science, Chinese Academy of Sciences, Beijing, China
  \end{minipage}
  \par\bigskip
  \begin{minipage}{0.7\linewidth}
    \centering
    \textsuperscript{2}Department of Basic Sciences in Engineering Sciences, University of Innsbruck, Innsbruck, Austria
  \end{minipage}
  \par\bigskip
  \begin{minipage}{0.7\linewidth}
    \centering
    \textsuperscript{3}CEMSE, King Abdullah University of Science and
    Technology, Saudi Arabia
  \end{minipage}
  \makeatother
\end{center}

\footnotetext[4]{Daniel F. Scharler greatly contributed to an earlier version of this manuscript. He tragically passed away on April 12, 2022 at the age of only~29.}

\begin{abstract}
  Motion polynomials (polynomials over the dual quaternions with nonzero
  real norm) describe rational motions. We present a necessary and sufficient
  condition for reduced bounded motion polynomials to admit factorizations into
  monic linear factors, and we give an algorithm to compute them. We can use those
  linear factors to construct mechanisms because the factorization corresponds
  to the decomposition of the rational motion into simple rotations or
  translations. Bounded motion polynomials always admit a factorization into
  linear factors after multiplying with a suitable real or quaternion polynomial.
  Our criterion for factorizability allows us to improve on earlier algorithms
  to compute a suitable real or quaternion polynomial co-factor.
 \end{abstract}

\section{Introduction}
\label{sec:introduction}

The factorization theory of polynomials over division rings was developed a century ago in \cite{gordon65,niven41,ore33}. It gained new attention recently because relations to mechanism theory were unveiled \cite{frischauf23, gallet16, hegedus13,juettler93, li15survey,li15darboux,li2016spatial,li19jsc,li17,rad18,scharler17,schicho2022and}. Quaternion polynomials parameterize rational spherical motions, that is, motions in $\SO$ with only rational trajectories. For describing motions in $\SE[2]$ or $\SE[3]$, one can use dual quaternion polynomials whose coefficients are subject to a quadratic condition (the ``Study condition'' that the norm polynomial is real and nonzero; cf.~Definition~\ref{def-motion}). The factorization theory of these so-called \emph{motion polynomials} turned out to be more complicated and, arguably, more interesting as well \cite{hegedus13,li19,li19jsc}.

The first publication on motion polynomial factorization is \cite{hegedus13}. There, the authors showed that a ``generic'' monic motion polynomial could be written in finitely many ways as a product of monic linear motion polynomials. They also related these factorizations to a decomposition of rational motions into coupled rotations and converted them into the construction of mechanical linkages. For instance, using different factorizations of a cubic motion polynomial, the authors in \cite{hegedus13} found a new family of paradoxically mobile closed-loop linkages with six revolute axes. Notice that the family contained the well-known Bricard Octahedra of type three \cite{gallet2021bricard} as a sub-class. Computational design issues were the topic of \cite{hegedus15}. In \cite{gallet16,li17}, the factorization of bounded motion polynomials was used for the construction of rational Kempe linkages. The potential for applications in architectural design is studied in \cite{li2020invertible}. Moreover, in \cite{liu2021structure}, it was used to synthesize single-loop variable-degree-of-freedom mechanisms.

It is easy to give examples of ``non-generic'' motion polynomials that admit infinitely many or no factorization. In \cite{li19jsc}, the authors showed that for any \emph{bounded} motion polynomial $M$ (cf.~Definition~\ref{def:generic-bounded} below; the name refers to the bounded trajectories of the underlying rational motion), there exists a real polynomial $S$ such that $MS$ admits a factorization. From the viewpoint of kinematics, this is an attractive property as $M$ and $MS$ describe the same rational motion. A variant of this theme can be found in \cite{li17} where the factorizability of $MT$ with a quaternion polynomial $T$ is studied to construct linkages to mechanically ``draw'' a prescribed curve.

The authors of \cite{li17,li19jsc} provide algorithms for computing the co-factors $S$ and $T$ but fail to characterize all cases where $S = T = 1$ is sufficient, that is, a direct factorization of $M$ is possible. While their algorithms will return $S = 1$ or $T = 1$ for many motion polynomials, a return value different from $1$ does not imply the non-existence of a factorization (cf.~the example in Section~\ref{sec:example}). In this article, we fill this gap for bounded motion polynomials by proving the following theorem.

\begin{thm}
  \label{th:bounded}
  Consider a bounded monic motion polynomial $M = cQ + \eps D \in
  \DH[t]$ with $c \in \R[t]$, $Q \in \H[t]$, $D \in \H[t]$ such that both $M$ and $Q$ have
  only constant real polynomial factors. Define $g$ to be the greatest common
  real polynomial divisor of the quaternion polynomials $c$, $\Cj{Q}D$, and
  $D\Cj{Q}$. The polynomial $M$ admits a factorization with monic linear motion
  polynomial factors if and only if $cg$ divides~${D}\Cj{D}$.
\end{thm}

We present two proofs of the theorem, both relying on the preliminaries in Section~\ref{sec:preliminaries}. The first proof is given in Section~\ref{sec:factorcondition}.
The necessity of the factorization condition is Proposition~\ref{prop:final-characterization}, and the sufficiency is Proposition~\ref{prop:reverse-final} there. 
The proof of the latter is divided into two major steps. First, we decompose the motion polynomial into a product of particular motion polynomials. They are not necessarily linear, but their norm polynomials are \emph{primary}, that is, powers of an irreducible (over $\R$) real polynomial.
Subsequently, we demonstrate that the factors of primary norms can be written as products of linear motion polynomials. This proof structure is also reflected in the structure of the paper and our algorithmic interpretation. After treating some aspects of the simple unbounded case in Section~\ref{sec:unbounded},
we decompose an arbitrary bounded motion polynomial into factors of a primary norm in Section~\ref{sec:primary-decomposition}. We consider the case of a primary norm in Section~\ref{sec:non-generic}.
The central results are Propositions~\ref{prop:singlepattern}--\ref{prop:alg_singlepattern}. The transition from the primary norm case to the general case is done in Section~\ref{sec:general}. The algorithm for computing factors with primary norm is Algorithm~\ref{alg:mgfactorization}, and the algorithm for decomposing polynomials with primary norm into linear factors is Algorithm~\ref{alg:singlepattern}. The combination of these algorithms yields a general method for factorizing bounded motion polynomials into linear factors, provided such a factorization exists. 

An alternative proof by induction is presented in Section~\ref{ap:simpler}. The second proof has the advantage of being
 simpler, but the first proof reveals more about the structure of the problem and also allows to extract a statement on factorizability with not-necessarily linear
 factors (Proposition~\ref{prop:alg_mgfactor}). These also have a kinematic relevance \cite{siegele23}.

\section{Preliminaries}
\label{sec:preliminaries}

The real associative algebra $\H$ of quaternions is commonly used to parameterize the special orthogonal group $\SO$. For a given quaternion $q = q_0 + q_1\qi + q_2\qj + q_3\qk$, we denote by $\Cj{q} \coloneqq q_0 - q_1\qi - q_2\qj - q_3\qk$ its \emph{conjugate} and by $\QNorm{q} \coloneqq q\Cj{q} = q_0^2 + q_1^2 + q_2^2 + q_3^2 \in \R$ its \emph{norm} (this is a convenient terminology in algebra; beware that in the other literature, $\sqrt{\QNorm{q}}$ is often called the norm of $q$). We say a quaternion $q$ is \emph{vectorial} if $q_0=0$, and in this case we have $\Cj{q}=-q$. Embed vector space $\R^3$ into $\H$ via $(x_1,x_2,x_3) \mapsto x \coloneqq x_1\qi + x_2\qj + x_3\qk$. Then the map
\begin{equation}
  \label{eq:SO}
  x \mapsto y = \frac{q x \Cj{q}}{q\Cj{q}}
\end{equation}
is the rotation of $x$ about $q_1\qi + q_2\qj + q_3\qk$ through the angle $\varphi = 2\arccos(q_0/\sqrt{\QNorm{q}})$, degenerating to the identity map for $q\in\mathbb{R}$. This assumes, of course, that $q \neq 0$. Note that nonzero scalar multiples of $q$ describe the same map \eqref{eq:SO} and we obtain an isomorphism from $\H \setminus \{0\}$ modulo the real multiplicative group $\R^\times$ to~$\SO$.

In order to construct a similar isomorphism into $\SE = \SO \ltimes \R^3$, we adjoin the element $\eps$ with $\eps^2 = 0$ to $\H$, where $\eps$ commutes with everything, thus arriving at the algebra $\DH$ of \emph{dual quaternions.} Formally, $\DH:=\H[\varepsilon]/(\varepsilon^2)$. A dual quaternion can be written as $h = p + \eps d$ where $p$ and $d$, the \emph{primal} and the \emph{dual parts}, are quaternions. The \emph{conjugate} of a dual quaternion $h$ is $\Cj{h} = \Cj{p} + \eps \Cj{d}$, and we will frequently make use of the formula $\Cj{(hk)} = \Cj{k}\Cj{h}$. The \emph{norm} $\QNorm{h} \coloneqq h\Cj{h} = \QNorm{p} + \eps(p\Cj{d} + d\Cj{p}) \in \D \coloneqq \R[\varepsilon]/(\varepsilon^2)$
turns out to be a dual number but it is customary in kinematics to require a real norm $\QNorm{h}$, i.e., $p\Cj{d} + d\Cj{p} = 0$. This is the famous \emph{Study condition} \cite{Study1891bewegungen}.

The \emph{$\eps$-conjugate} of $h$ is $h_\eps \coloneqq p - \eps d$. The embedding of $\R^3$ into $\DH$ is now done via $(x_1,x_2,x_3) \mapsto 1 + \eps(x_1\qi + x_2\qj + x_3\qk) = 1 + \eps x$, whence  for $h\Cj{h}\in \R \setminus \{0\}$ the map
\begin{equation}
  \label{eq:SE}
  1 + \eps x \mapsto 1 + \eps y = \frac{h_\eps (1 + \eps x) \Cj{h}}{h\Cj{h}}
\end{equation}
describes an element of $\SE$ and gives an isomorphism
\begin{equation}
  \label{eq:dh-isomorphism}
  \{ h \in \DH \mid h\Cj{h} \in \R \setminus \{0\}\} / \R^\times \to \SE,
\end{equation}
cf.~\cite{husty12}.

We now consider the algebra $\DH[t]$ of polynomials in the indeterminate $t$ with coefficients in $\DH$ and the sub-algebra $\H[t]$ with quaternion coefficients. Multiplication is defined by the requirement that $t$ commutes with all dual quaternions. This is only one possibility among many \cite{ore33}. It is appropriate in kinematics because we think of $t$ as a real motion parameter and $\R$ is in the center of $\DH$. For $M \in \DH[t]$ we denote by $\Cj{M}$ and $M_\eps$ the polynomials obtained by conjugating or $\eps$-conjugating its coefficients, respectively, and by $\QNorm{M}:=M\Cj{M} \in \D[t]$ its \emph{norm polynomial}. There exist unique polynomials $P$, $D \in \H[t]$ such that $M = P + \eps D$. They are called the \emph{primal} and the \emph{dual part} of $M$, respectively.

\begin{defn} \label{def-motion}
  A polynomial $M \in \DH[t]$ is called a \emph{motion polynomial,} if $M\Cj{M} \in \R[t] \setminus \{0\}$.
\end{defn}

The motivation for this definition is \eqref{eq:dh-isomorphism} and the possibility for a parametric version of map~\eqref{eq:SE}:
\begin{equation}
  \label{eq:rational-motion}
  1 + \eps x \mapsto \frac{M_\eps (1 + \eps x) \Cj{M}}{M\Cj{M}} =
  1 + \eps \frac{Px\Cj{P} + P\Cj{D} - D\Cj{P}}{P\Cj{P}}.
\end{equation}
The right-hand side of \eqref{eq:rational-motion} is a rational parametric curve. Any rigid body motion with the property of having exclusively rational trajectories can be described by a suitable motion polynomial which is unique up to multiplication with real polynomials \cite{juettler93}.

The main contribution of this article is a simple necessary and sufficient condition for a motion polynomial $M$ of degree $n$ to have a decomposition into monic linear motion polynomial factors, i.e.,
\begin{equation}
  \label{eq:mp-factorization}
  M = m_n(t - h_1)(t - h_2) \cdots (t - h_n)
\end{equation}
with the leading coefficient $m_n \in \DH$ and linear motion polynomials $t - h_i$, $i \in \{1,2,\ldots,n\}$. A linear motion polynomial $t-h$ parametrizes a rotation or translation. The rotation axis or the translation direction are defined by $h$ \cite{hegedus13}. It is no loss of generality, possibly after a rational re-parametrization, to assume that the leading coefficient $m_n$ is invertible \cite{li16}. Since multiplying with $m_n^{-1}$ does not change relevant kinematic or algebraic properties, we will only consider \emph{monic motion polynomials} in the following. Factorizations of form~\eqref{eq:mp-factorization} are interesting as they correspond to decompositions of rational motion \eqref{eq:rational-motion} into coupled rotations around fixed axes or, in exceptional cases, translations along fixed lines. This is the algebraic foundation for the linkage constructions of \cite{hegedus13,hegedus15,li2020invertible,li17,liu2021structure}.

We denote by $\mrpf(A,B)$ the greatest common real monic polynomial divisor of $A$ and~$B$. In addition, 
we define $\mrpf(M):= \mrpf(M,M)$ and $\mrpf(0):=1$. In case of $\mrpf(M) = 1$, we say
that $M$ is \emph{reduced.}

\begin{defn}
  \label{def:generic-bounded}
  A monic reduced motion polynomial $M = P + \eps D$ is \emph{generic} if $\mrpf(P) = 1$. It is called \emph{bounded} if $\mrpf(P)$ has no real zero, and \emph{unbounded} otherwise.
\end{defn}

From \eqref{eq:rational-motion} we see that the trajectories of a bounded motion polynomial are bounded rational curves while they are unbounded otherwise.

We also see that a criterion for the factor $t - h_i$ in \eqref{eq:mp-factorization} to describe a translation is the reality of its primal part. Linear factors of this type exist only for unbounded $M$. More generally, a motion polynomial $M = P + \eps D$ is called \emph{translational} if $P \in \R[t]$.

The usual division in $\DH[t]$ with remainders is possible if the divisor has an invertible leading coefficient \cite[Theorem~1]{li19}. More precisely, given $A$, $B \in \DH[t]$ with the leading coefficient of $B$ being invertible, there exist unique polynomials $Q_l$, $Q_r$, $R_l$, and $R_r$ such that
\begin{equation*}
  A = Q_lB + R_l,
  \quad A = BQ_r + R_r,
  \quad\text{and}\quad \deg R_l,\deg R_r<\deg B.
\end{equation*}
They are called \emph{left/right quotient} and \emph{left/right remainder,}
respectively. We denote them by $\lquo(A,B) \coloneqq Q_l$, $\rquo(A,B)
\coloneqq Q_r$, $\lrem(A,B) \coloneqq R_l$, and $\rrem(A,B) \coloneqq R_r$.
Moreover, if $R_l = 0$ ($R_r = 0$), we say that $Q_l$ ($Q_r$) is a left (right)
\emph{factor}, or \emph{divisor}, of $A$. For nonzero quaternion polynomials $A$ and $B$, the leading coefficient is always invertible; the unique monic left/right common divisor of the maximal degree for $A$ and
$B$ is denoted by $\lgcd(A, B)$ and $\rgcd(A, B)$, respectively. It is a 
quaternion polynomial and can be computed by a non-commutative version of the
Euclidean algorithm; see~\cite{li16,ore33} for more details. We say that a real polynomial $F\in\R[t]$ \emph{divides} a motion polynomial $M\in\DH[t]$ and write $F | M$ if $F$ is a factor of $M$.

We are going to use the following two folklore results which are related to references but are stated differently; cf.~\cite[Lemma~2.2]{li23abc}, \cite[Proposition~2.1]{cheng16}, \cite[Lemma~5]{li19jsc}, and \cite[Lemma~1]{li23}.
It is worth pointing out that the statement on quaternion polynomials is related to the one on complex quaternions if one replaces the parameter $t$ by the complex imaginary unit and takes the quadratic polynomial $N=t^2+1$, w.l.o.g. 

\begin{lem}[ABC-lemma]
  \label{lem:a-b-c-lemma}
  If a quadratic irreducible polynomial $N \in \R[t]$ divides the product $ABC$ of polynomials $A,B,C\in \H[t]$ and $N$ divides $B\Cj{B}$, then $N$ divides $AB$ or $BC$. 
\end{lem}
\begin{proof}
     Assume that $N$ does not divide $B$; otherwise, there is nothing to prove. Using division with remainders, we find $A_1,B_1,B_2,C_1\in\H[t]$ and $A_2,C_2\in\H$ such that $B=B_1N+B_2$, $A=A_1\Cj{B_2}+A_2$, $C=\Cj{B_2}C_1+C_2$, and $B_2\ne 0$. Since $N$ divides $B\Cj{B}$, it divides $B_2\Cj{B_2}$ as well. Since $N$ divides $ABC$, it divides $A_2B_2C_2$ as well. The latter polynomial has a degree at most $1$, hence it vanishes. Since $B_2\ne 0$, it follows that $A_2=0$ or $C_2=0$, hence $N$ divides $AB$ or $BC$.
\end{proof}

\begin{lem}[AB-Lemma]
  \label{lem:g-a-b-lemma}
  If a monic polynomial $F\in \R[t]$ divides the product $AB$ of polynomials 
  $A,B\in \H[t]$ but 
  $\mrpf(F,A)=\mrpf(F,B)=1$, then $\rgcd(F,A)=\lgcd(F,B)^*$ and $F=\rgcd(F,A)\lgcd(F,B)$.
\end{lem}
\begin{proof}
     The polynomial $F$ has no real roots, otherwise, it would have a common real root with $A$ or $B$, contradicting $\mrpf(F,A)=\mrpf(F,B)=1$. Let $N$ be a real irreducible quadratic factor of $F$. Since $N\,|\,F\,|\,AB$, it follows that $N\,|\,A\Cj{A}B\Cj{B}$, hence $N\,|\,A\Cj{A}$ or $N\,|\,B\Cj{B}$. Assume that $N\,|\,A\Cj{A}$ without loss of generality. Using division with remainders, we find $Q,R,B_1\in\H[t]$ and $r\in\H$ such that $A=QN+R$ and $B=\Cj{R}B_1+r$. Here $R\ne 0$ because $\mrpf(F,A)=1$. Since  $N\,|\,A\Cj{A}$, it follows that $R\Cj{R}=pN$ for some nonzero $p\in\mathbb{R}$. Since $N\,|\,AB$, it follows that $N\,|\,Rr$, hence $r=0$. The polynomials $A_1:=AR^{-1}=Q\Cj{R}/p+1$, $B_1=(\Cj{R})^{-1}B$, and $F_1:=F/N$ still satisfy the assumptions of the lemma. Applying the same construction inductively, we get $A_n,B_n,R_n\in\H[t]$ such that $A=A_nR_n$, $B=\Cj{R_n}B_n$, $F=R_n\Cj{R_n}$, and $R_n$ is monic.

     It remains to show that $R_n=\rgcd(F,A)$ and $\Cj{R_n}=\lgcd(F,B)$ (cf.~\cite[Lemma~3]{li16}). Indeed, by the extended Euclidean algorithm, $\rgcd(F,A)=XA+YF$ for some $X,Y\in\H[t]$. Hence $\rgcd(F,A)=SR_n$ for some monic $S\in\H[t]$. Since $F=\Cj{R_n}R_n$ and $A=A_nR_n$ are right-divisible by $\rgcd(F,A)=SR_n$, it follows that $\Cj{R_n}$ and $A_n$ are right-divisible by $S$.
     Then $R_n$ is left-divisible by $\Cj{S}$, hence  
     $F$ and $A$ are divisible by $S\Cj{S}$. Since $\mrpf(F,A)=1$, it follows that $S=1$ and $R_n=\rgcd(F,A)$. Analogously, 
     $\Cj{R_n}=\lgcd(F,B)$.
\end{proof}

Let us conclude this section with a definition of the product symbol. We set $\prod_{i=l}^{u} r(i)=r(l)r(l+1)\dots r(u)$ for $l\le u$ and $\prod_{i=l}^{u} r(i)=1$ when $l>u$, where $r(i)$ is a function in $i$.

\section{Conditions for Existence of Factorizations with Linear Factors}
\label{sec:factorcondition}

\subsection{Factorizations of Unbounded Motion Polynomials}
\label{sec:unbounded}

For an unbounded motion polynomial $M$, the existence of a factorization is not guaranteed, not even after multiplication with a real polynomial $S \in \R[t]$ or a quaternion polynomial $T \in \H[t]$ \cite[Example~13]{li19}. However, there is a dense set of unbounded motions polynomials that admit a factorization:

\begin{thm}[{\cite[Theorem~7]{li19}}]
  \label{th:unbounded-sufficient}
  For any unbounded motion polynomial $M = P + \eps D \in \DH[t]$ with $P$, $D \in  \H[t]$ such that $P$ has no linear real factors of multiplicity two or greater,  there exists $S \in \R[t]$ such that $MS$ admits a factorization with linear motion polynomial factors.
\end{thm}

We complement Theorem~\ref{th:unbounded-sufficient} by the following result. 

\begin{prop}
  \label{prop:unbounded-necessary}
  If a reduced unbounded motion polynomial $M = P + \eps D \in \DH[t]$ with $P$, $D \in \H[t]$ admits a factorization into linear motion polynomials, then $P$ has no linear real factors of multiplicity two.
\end{prop}

\begin{proof}
  Assume $\mrpf(P)$ has a linear real factor $p$ of multiplicity two. We can factor $M$ into $M_1M_2$, where $M_1 = pP_1+\eps D_1$ and $M_2 = pP_2+\eps  D_2$ by expanding the product of linear factors into two parts. But the product   $M_1M_2=p^2P_1P_2+\eps p (D_1P_2+P_1D_2)$ is not reduced. We have a contradiction.
\end{proof}

In what follows, we consider only bounded motion polynomials $M = P + \eps D$, where $P = cQ$ with monic $c \in \R[t]$ and $Q \in \H[t]$ such that $\mrpf(Q) = 1$ and $c$ has no real zeros.

\subsection{Factorizations of Generic Motion Polynomials}
\label{sec:generic}

For a generic motion polynomial, the factorization exists and there is a simple algorithm to construct it:

\begin{thm}[{\cite[Theorem~1]{hegedus13}}]
  \label{thm:alg_galg} 
  For a generic monic motion polynomial $M\in \DH[t]$ of degree~$n$,
Algorithm~\ref{galg} outputs a list $[L_1,L_2,\ldots,L_n]$ of monic linear motion polynomials such that
    $M=L_1L_2 \cdots L_n$.
\end{thm}

Algorithm~\ref{galg} shows the algorithm of \cite{hegedus13} for the factorization of a generic motion polynomial $M$. Let us explain the notation there. If $F \in \R[t]$ is an irreducible quadratic factor of the norm polynomial $\QNorm{M}$, the remainder $R = r_1t + r_0$ when dividing $M$ by $F$ has, in general, a unique right zero $h = -r_1^{-1}r_0$. More precisely, the zero is unique if and only if the leading coefficient $r_1$ is invertible. In this case, we denote it by $\czero(M,F) \coloneqq h$. A necessary and sufficient condition for $r_1$ to be invertible is that $F$ does not divide the primal part $P$ of $M$, i.e., $\mygcd(P,F) = 1$. Because $M$ is generic, this condition is always fulfilled and the computation of $h$ in Step~\ref{galg:h} of Algorithm~\ref{galg} will always work. In Step~\ref{galg:concat}, two lists are concatenated and the procedure is called recursively. Non-uniqueness of the factorization comes from the possibility of selecting the quadratic factor $F$ in Step~\ref{galg:pick-quadratic-factor}.

\begin{algorithm}
  \caption{\texttt{GFactor}: Factorization algorithm for generic motion polynomials}
  \label{galg}
  \begin{algorithmic}[1]
    \Require A generic monic motion polynomial $M \in \D\H[t]$ of degree~$n$.
    \Ensure A list $[L_1,\ldots,L_n]$ of monic linear motion polynomials such that
    $M=L_1\cdots L_n$.
\If{$\deg M = 0$}
      \State \Return $[\,]$ \Comment{Empty list.}
    \EndIf
    \State $F \adef$ a quadratic real factor $F$ of the norm polynomial $\QNorm{M} \in \R[t]$
    \label{galg:pick-quadratic-factor}
    \State $h \adef \czero(M, F)$
    \label{galg:h}
    \State $M \adef \lquo(M, t-h)$
    \State \Return $\lconcat{\mathtt{GFactor}(M)}{t-h}$
    \label{galg:concat}
  \end{algorithmic}
\end{algorithm}

We complement Theorem~\ref{thm:alg_galg} by the following result.

{\begin{cor}
  \label{cor:alg_galg}
  Let $M = P + \eps D \in \DH[t]$ be a monic motion polynomial, $G \in \R[t]$ be a monic  
  polynomial such that $G| \QNorm{M}$ and $\mrpf(G,P)=1$. Then there is a motion polynomial $M_2$ that admits a factorization into monic linear factors such that $G=M_2\Cj{M_2}$ and $M=M_1M_2$ for some motion polynomial $M_1$. \end{cor}
\begin{proof}
The polynomial $G$ has no real roots because $G| \QNorm{P}$ and $\mrpf(G,P)=1$.
The linear factors are obtained using Algorithm~\ref{galg} by the recursive procedure where each quadratic real factor of $G$ gives a linear factor. \end{proof}
}

Beware that if a non-generic motion polynomial $M$ has a decomposition into monic linear motion polynomial factors, $M_1$ need not have one.

\subsection{Decomposition of Bounded Motion Polynomials into Factors with Primary Norm}

\label{sec:primary-decomposition} 

For a non-generic motion polynomial $M$, the recent results from \cite{li17,li19jsc} ensure the existence of factorizations after multiplying $M$ with a suitable real or quaternion polynomial:
\begin{thm}[\cite{li17,li19jsc}]
For any bounded monic motion polynomial $M = P + \eps D \in \DH[t]$ with
  $P$, $D \in \H[t]$:
  \begin{itemize}
  \item There exists a polynomial $S \in \R[t]$ of degree $\deg S \le
    \deg\mrpf(P)$ such that $MS$ admits a factorization with linear motion polynomial factors.
  \item If $\mygcd(P,\QNorm{D}) = 1$ then there exists $T \in \H[t]$ of
    degree $\deg T = \frac{1}{2}\deg\mygcd(P)$ such that $MT$ admits a
    factorization with linear motion polynomial factors.
  \end{itemize}
\end{thm}

The algorithm of \cite{li19jsc} for computing the co-factor $S$ is too complicated
  to be discussed here. We confine ourselves to a simple example.

\begin{example}
  \label{ex:multiplication-trick}
  Consider the polynomial $M = t^2 + 1 + \eps\qi$. It admits no factorization with monic linear motion polynomial factors (this follows from Theorem~\ref{th:bounded} and can also be checked by a direct computation~\cite[Example 1]{li19jsc}). But with $S = t^2+1$ and $T = t - \qk$, we have
  \begin{align*}
    MS &=
    (t+\tfrac{3}{5}\qj-\tfrac{4}{5}\qk)
    (t - \tfrac{3}{5}\qj+\tfrac{4}{5}\qk + \eps(\tfrac{2}{5}\qj+\tfrac{3}{10}\qk))
    (t - \tfrac{3}{5}\qj+\tfrac{4}{5}\qk - \eps(\tfrac{2}{5}\qj + \tfrac{3}{10}\qk))
    (t+\tfrac{3}{5}\qj-\tfrac{4}{5}\qk),\\
MT &= (t + \qk)
    (t - \qk - \tfrac{1}{2}\eps\qj)
    (t - \qk + \tfrac{1}{2}\eps\qj).
  \end{align*}
\end{example}

The above results state that a factorization of a bounded motion polynomial into linear factors can be guaranteed but it might be necessary to multiply it with a real polynomial (which does not change the underlying motion) or with a quaternion polynomial (which does not change the trajectory of the origin). One can use the factorization of motion polynomials to construct linkages with a prescribed bounded rational trajectory as \cite{gallet16, li17}. However, the number of joints will increase with the degree of $S$ or $T$. Therefore, it is desirable to have co-factors $S$ or $T$ of low degree, possibly even of degree zero. This paper focuses on precisely this situation. We characterize bounded motion polynomials that admit a factorization with monic linear factors without requiring prior multiplication with a real or quaternion polynomial.

Our approach has two major steps. In the first step, we provide a factorization $M = M_1 \cdots M_r$ into not necessarily linear factors $M_i$ of \emph{primary} norm, i.e., $\QNorm{M_i} = N_i^{n_i}$, where $N_i \in \R[t]$ is an irreducible quadratic polynomial and $n_i$ is a positive integer. In the second step, we obtain factorizations of each polynomial $M_i$ of primary norm into monic linear factors.

Recall that a motion polynomial is \emph{translational} if its primal part is a real polynomial.

\begin{prop}
  \label{prop:alg_translation}
  For any monic translational motion polynomial $M = f_1 f_2 + \eps D \in \DH[t]$ such that $D \in \H[t]$ and $f_1$, $f_2 \in \R[t]$ are two coprime monic real polynomials, Algorithm~\ref{alg:factortranslation} outputs translational motion polynomials $M_1 = f_1 + \eps D_1$, $M_2 = f_2 + \eps D_2$ such that $M = M_1M_2$.
\end{prop}

\begin{proof}
Let $d_1$ and $d_2$ be two real polynomials such that $f_1d_1+f_2d_2 = 1$, for example, those given by the extended Euclidean algorithm.
Then multiplying both side by $D$ yields $D = f_1 d_1 D + f_2 d_2 D$. We divide $d_1 D$ and $d_2 D$ by $f_2$ and $f_1$, respectively, and obtain the quotients $H_1$, $H_2$, and the remainders $D_1$, $D_2 \in \H[t]$ such that $d_1 D = f_2 H_2 + D_2$ and $d_2 D = f_1 H_1 + D_1$ whence
  \begin{equation*}
    D = f_1 (f_2 H_2 + D_2) + f_2 (f_1 H_1 + D_1) = f_1 f_2 (H_1 + H_2) + f_1 D_2 + f_2 D_1.
  \end{equation*}
  Comparing coefficients on both sides we conclude that $H_1 + H_2 = 0$ by monicity of $M$, i.e., $\deg(D) < \deg(f_1 f_2)$. Therefore $D = f_1D_2 + f_2D_1$. With this we have
  \begin{equation*}
    (f_1 + \eps D_1) (f_2 + \eps D_2) = f_1 f_2 + \eps(f_1 D_2 + f_2 D_1) = f_1 f_2 + \eps D = M.
  \end{equation*}
  Since $D$, $H_1$, and $H_2$ are vectorial, and $\deg(D_i) < \deg(f_i)$ for $i =
  1$, $2$ by the definition of polynomial division, it follows that $f_1 + \eps D_1$ and $f_2
  + \eps D_2$ are monic motion polynomials.
\end{proof}

Notice that $M$ needs not be reduced or bounded.

\begin{algorithm}
  \caption{Factorization algorithm for translational motion polynomials}
  \label{alg:factortranslation}
  \begin{algorithmic}[1]
    \Require A monic translational 
    motion polynomial $M = f_1f_2 + \eps D \in \DH[t]$ with $f_1,f_2\in\R[t]$ and $\mygcd(f_1,f_2)=1$.
    \Ensure Two motion polynomials $M_1 = f_1 + \eps D_1$, $M_2 = f_2 + \eps D_2$ such that $M = M_1M_2$.
    \State Find real polynomials $d_1$, $d_2$ with $f_1d_1+f_2d_2=1$ by the extended Euclidean algorithm.
\State $D_1 \adef \rrem(d_1D, f_2)$, $D_2 \adef \rrem(d_2D, f_1)$
    \State $M_1 \adef f_1+\eps D_1$
    \State $M_2 \adef f_2+\eps D_2$
    \State \Return $M_1$, $M_2$
  \end{algorithmic}
\end{algorithm}

Now we compute $M_1$, $M_2$, \ldots, $M_r$ for non-translational polynomials $M$.

\begin{algorithm}
  \caption{\texttt{MGFactor}: Algorithm for decomposing motion polynomials into
    factors of primary norm}
  \label{alg:mgfactorization}
  \begin{algorithmic}[1]
    \Require A bounded monic reduced motion polynomial $M=P+\eps D \in \D\H[t]$.
\Ensure A list $[M_1,M_2,\ldots,M_r]$ of monic motion polynomials such that
    $M=M_1 M_2 \cdots M_r$ and $\QNorm{M_1}= N_1^{n_1}$, \dots, $\QNorm{M_r}= N_r^{n_r}$ for some pairwise coprime real irreducible quadratic polynomials $N_1$,     $N_2$, \ldots, $N_r$ and positive integers $n_1$, $n_2$, \ldots, $n_r$.
    \If{$\deg M = 0$}
      \State \Return $[]$ \Comment{Empty list.}
    \EndIf
    \State $N\adef$ quadratic, real factor of $\QNorm{M} \in \R[t]$ of
    multiplicity $n$
    \State $Q \adef \frac{P}{\mrpf(P)}$, $c_2 \adef \mygcd(P, N^{n})$, and $c_1 \adef \mrpf(P)/c_2$
    \State $Q_2 \adef \rgcd(N^{n}/c_2^2,Q)$ and $Q_1 \adef \lquo(Q,Q_2)$
    \State $M_t \adef c_1c_2\QNorm{Q_1}\QNorm{Q_2}+ \eps \Cj{Q}_1D\Cj{Q}_2$
    \State compute the factorization $M_t = (c_1\QNorm{Q_1}+\eps
    D_1)(c_2\QNorm{Q_2}+\eps D_2)$ \Comment Using Algorithm~\ref{alg:factortranslation}.
    \State $M' \adef c_1Q_1 +\eps \frac{D_1Q_1}{\QNorm{Q_1}}$, $M_r \adef c_2Q_2 +\eps \frac{D_2Q_2}{\QNorm{Q_2}}$\\
    \Return $[\mathtt{MGFactor}(M'), M_r]$
  \end{algorithmic}
\end{algorithm}

{
\begin{prop}
  \label{prop:alg_mgfactor}
   For a bounded monic reduced motion    polynomial $M \in \D\H[t]$, 
Algorithm~\ref{alg:mgfactorization} outputs a list $[M_1,M_2,\ldots,M_r]$ of monic motion polynomials such that
    $M=M_1 M_2 \cdots M_r$ and $\QNorm{M_1}= N_1^{n_1}$, \dots, $\QNorm{M_r}= N_r^{n_r}$ for some pairwise coprime real irreducible quadratic polynomials $N_1$,     $N_2$, \ldots, $N_r$ and positive integers $n_1$, $n_2$, \ldots, $n_r$. 
\end{prop}
\begin{proof} Let $\QNorm{M}=N_1^{n_1}N_2^{n_2} \cdots N_r^{n_r}$ be a decomposition into monic real irreducible factors. We perform induction on $r$. If $r=0$, then the algorithm outputs
  the empty list and stops. Assume that for an integer $r-1$, the algorithm
  outputs a list of $r-1$ elements as needed.

In Step~4, let $N=N_r$ without loss of generality. Using the definitions in Steps~5--6 and applying Lemma~\ref{lem:g-a-b-lemma} to $A:=Q$, $B:=\Cj{Q}$, and $F:=N^n/c_2^2$, we get
\begin{equation*}
P= c_1 Q_1 c_2Q_2,\quad
    \QNorm{c_1 Q_1} = \prod_{i=1}^{r-1} N_i^{n_i},\quad
    \QNorm{c_2 Q_2} = N_r^{n_r}
  \end{equation*}
  and
\begin{align*}
   M = \frac{\QNorm{Q_1}}{\QNorm{Q_1}} M \frac{\QNorm{Q_2}}{\QNorm{Q_2}} 
    &= \frac{Q_1}{\QNorm{Q_1}} \Cj{Q}_1 (c_1 Q_1 c_2 Q_2 + \eps D) \Cj{Q}_2 \frac{Q_2}{\QNorm{Q_2}}\\
    &= \frac{Q_1}{\QNorm{Q_1}} (c_1 \QNorm{Q_1} c_2 \QNorm{Q_2} + \eps \Cj{Q}_1 D \Cj{Q}_2) \frac{Q_2}{\QNorm{Q_2}}.
  \end{align*}    
The condition $\mygcd(c_1 \QNorm{Q_1}, c_2 \QNorm{Q_2}) = 1$ allows us to apply Proposition~\ref{prop:alg_translation} to the middle factor, and we obtain vectorial quaternion polynomials $D_1$, $D_2 \in \H[t]$ such that
  \begin{equation*}
    M = \frac{Q_1}{\QNorm{Q_1}} (c_1 \QNorm{Q_1} + \eps D_1) (c_2 \QNorm{Q_2} + \eps D_2) \frac{Q_2}{\QNorm{Q_2}}
      = \underbrace{\left(c_1 Q_1 + \eps \frac{Q_1 D_1}{\QNorm{Q_1}}\right)}_{M'} \underbrace{\left(c_2 Q_2 + \eps \frac{D_2 Q_2}{\QNorm{Q_2}}\right)}_{M_r}.
  \end{equation*}
  {By Proposition~\ref{prop:alg_translation}, we have}
\begin{align*}
    D_1 &= d_2 \Cj{Q}_1 D \Cj{Q}_2 - H_1 c_1 \QNorm{Q_1} = \Cj{Q}_1 D \Cj{Q}_2 d_2 - \Cj{Q}_1Q_1 H_1 c_1,\\
    D_2 &= d_1 \Cj{Q}_1 D \Cj{Q}_2 - H_2 c_2 \QNorm{Q_2} = d_1 \Cj{Q}_1 D \Cj{Q}_2 - H_2 c_2 Q_2\Cj{Q}_2
  \end{align*}
  for suitable $H_1$, $H_2 \in \H[t]$ and $d_1,d_2\in\R[t]$. The polynomial $\Cj{Q}_1$ is a left factor of $D_1$, and $\Cj{Q}_2$ is a right factor of $D_2$. Hence, $Q_1 D_1$ is divisible by $\QNorm{Q_1}$ and $D_2 Q_2$ is divisible by $\QNorm{Q_2}$. Therefore, both $M'$ and $M_r$ are polynomials. Obviously, they are monic and bounded. The motion polynomial property is derived by
  \begin{align*}
    \QNorm{M'} &= \Cj{\left(c_1 Q_1 + \eps \frac{Q_1 D_1}{\QNorm{Q_1}}\right)} \left(c_1 Q_1 + \eps \frac{Q_1 D_1}{\QNorm{Q_1}}\right) = c_1^2 \QNorm{Q_1} + \eps c_1 \underbrace{(D_1 + \Cj{D}_1)}_{=0} = \prod_{i=1}^{r-1} N_i^{n_i}, \\
    \QNorm{M_r} &= \left(c_2 Q_2 + \eps \frac{D_2 Q_2}{\QNorm{Q_2}}\right) \Cj{\left(c_2 Q_2 + \eps \frac{D_2 Q_2}{\QNorm{Q_2}}\right)} = c_2^2 \QNorm{Q_2} + \eps c_2  \underbrace{(D_2 + \Cj{D}_2)}_{=0} = N_r^{n_r},
  \end{align*}
  since $D_1$ and $D_2$ are vectorial polynomials. A recursive application of above procedure on $M'$ finishes the proof.
\end{proof}

}

\subsection{Factorizations of Bounded Motion Polynomials with Primary Norm}
\label{sec:non-generic}

In this subsection, we focus on polynomials $M = cQ + \eps D$ whose norm polynomials have only one irreducible factor $N$, i.e., $\QNorm{M} = c^2\QNorm{Q} = N^n$. In what follows we assume that $c\in\R[t]$ and $Q,D\in\H[t]$. Once we obtain conditions for the factorizability of such polynomials $M$, we are going to extend the results to arbitrary motion polynomials using Proposition~\ref{prop:alg_mgfactor}.

We start with an auxiliary factorization.

\begin{lem}
  \label{lem:factorizationpattern}
  Consider a bounded, monic, reduced motion polynomial $M = cQ + \eps D \in \D\H[t]$ of degree $n$ with monic $c\in\R[t]$, $\mrpf(Q) = 1$, and $\QNorm{M} = c^2\QNorm{Q} = N^n$ for some irreducible, quadratic polynomial $N \in \R[t]$. If $M$ admits a factorization into monic linear motion polynomial factors, then there exist quaternion polynomials $Q_L$, $Q_R$, $D_L$, $D_R$, $D' \in \H[t]$  such that  $M = (Q_L + \eps D_L) (c + \eps D') (Q_R + \eps D_R)$, where $Q_L + \eps D_L$, $c + \eps D'$, and $Q_R + \eps D_R$ are products of monic linear motion polynomials.
\end{lem}

\begin{proof}
  Let $M = \prod_{u=1}^n (t - p_u - \eps d_u)$ be a factorization of $M$. If
  $\deg(c) = 0$, we can choose $Q_L + \eps D_L = \prod_{u=1}^v (t - p_u - \eps
  d_u)$, $Q_R + \eps D_R = \prod_{u=v+1}^n (t - p_u - \eps d_u)$ and $D' = 0$
  for any integer $v \in \{ 1,\ldots ,n-1 \}$. Hence, we may assume that
  $\deg(c) > 0$. By the assumption, we have $N^n = \QNorm{M} = c^2 \QNorm{Q}$
  and there exists $m$ such that $N^m = c$.
  Moreover, $c$ divides the primal part $\prod_{u=1}^n(t-p_u)$ of $M$. Let $v
  \in \{ 1,\ldots ,n \}$ be the smallest integer such that $N$ divides $P_1
  \coloneqq \prod_{u=1}^v (t-p_u)$. 
  By Lemma~\ref{lem:a-b-c-lemma} applied to $\prod_{u=1}^{v-2} (t-p_u)$, $t-p_{v-1}$, and $t-p_v$, we get $p_{v-1} = \Cj{p}_v$. Further,
$N$ cannot divide $P_2 \coloneqq
  \prod_{u=v+1}^n (t-p_u)$ because otherwise
  \begin{equation*}
    M = \prod_{u=1}^v(t-p_u-\eps d_u) \prod_{u=v+1}^n(t-p_u-\eps d_u) = (P_1 + \eps D_1)(P_2 + \eps D_2) = P_1 P_2 + \eps(P_1 D_2 + D_1 P_2)
  \end{equation*}
  for some quaternion polynomials $D_1$, $D_2 \in \H[t]$ and $N$ divides $M$, a
  contradiction to $M$ being reduced. An inductive application of  Lemma~\ref{lem:a-b-c-lemma}
yields
  \begin{equation*}
    p_{v-2} = \Cj{p}_{v+1},\
    p_{v-3} = \Cj{p}_{v+2}, \ldots,\
    p_{v-m} = \Cj{p}_{v+m-1}
  \end{equation*}
  and $c = \prod_{u=v-m}^{v+m-1}(t-p_u)$. We define $Q_L$, $Q_R$, $D_L$, $D_R$,
  and $D'$ such that
  \begin{equation*}
    Q_L + \eps D_L = \prod_{u=1}^{v-m-1}(t-p_u-\eps d_u), \quad
    c + \eps D' = \prod_{u=v-m}^{v+m-1}(t-p_u-\eps d_u), \quad
    Q_R + \eps D_R = \prod_{u=v+m}^n(t-p_u-\eps d_u).
\end{equation*}
  Therefore, the statement follows.
\end{proof}

\begin{cor}\label{cor:weakcor}
  Consider a bounded, monic, reduced motion polynomial $M = c + \eps D \in \DH[t]$ of degree 
 $n$ with $\QNorm{M} = c^2 = N^n$ for some irreducible, quadratic polynomial $N \in \R[t]$. If $M$ admits a factorization with monic linear motion polynomial factors, then it is of the form
  \begin{equation*}
    M =(t-p_1-\eps d_1) \cdots (t-p_n-\eps d_n)(t-\Cj{p}_n-\eps d'_n) \cdots (t-\Cj{p}_1-\eps d'_1),
  \end{equation*}
  where $p_i$, $d_i$, $d'_i \in \H$ are quaternions for $i=1, 2, \ldots ,n$.
\end{cor}

Lemma~\ref{lem:factorizationpattern} allows us to formulate a necessary
condition for factorizability of translational motion polynomials.

\begin{lem}\label{lem:weaklemma}
  Consider a bounded, monic, reduced motion polynomial $M = c + \eps D \in
  \DH[t]$ of degree $n$ with $\QNorm{M} = c^2 = N^n$ for some monic irreducible
  quadratic polynomial $N \in \R[t]$. If $M$ admits a
  factorization with monic linear motion polynomial factors, then $c$ divides
  $\QNorm{D}$.
\end{lem}
\begin{proof}
  Clearly, $n$ is even.
  We give a proof by induction on $n$. For $n=2$, the factorization has form \begin{equation*}
    M = c + \eps D = N + \eps D = (t-p-\eps d_1)(t-\Cj{p}-\eps d_2)
  \end{equation*}
  by Corollary~\ref{cor:weakcor} and therefore $D= -d_1(t-\Cj{p})-(t-p)d_2$. The
  polynomial $(t-p-\eps d_1)$ is a motion polynomial,
  i.e. $(t-p)\Cj{d}_{1} + d_{1}(t-\Cj{p}) = 0$. The norm of the dual
  part $D$ simplifies to
  \begin{align*}
    \QNorm{D} &= \Cj{D}D = (-(t-p)\Cj{d}_1-\Cj{d}_2(t-\Cj{p})) (-d_1(t-\Cj{p})-(t-p)d_2) \\
              &= (t-p)\Cj{d}_1d_1(t-\Cj{p}) + \Cj{d}_2(t-\Cj{p})d_1(t-\Cj{p}) + (t-p)\Cj{d}_1(t-p)d_2 + \Cj{d}_2(t-\Cj{p})(t-p)d_2 \\
              &= N (\QNorm{d_1} -\Cj{d}_2\Cj{d}_1 - d_1d_2 + \QNorm{d_2}) = c (\QNorm{d_1} -\Cj{d}_2\Cj{d}_1 - d_1d_2 + \QNorm{d_2}).
  \end{align*}

  To perform the induction step, assume $M$ has degree $n+2$. Again, by
  Corollary~\ref{cor:weakcor} we have
  \begin{equation*}
    M = c + \eps D = N^{\frac{n+2}{2}} +\eps D = (t-p-\eps d_1)(N^{\frac{n}{2}}+\eps D')(t-\Cj{p}-\eps d_2),
  \end{equation*}
  where $N^{\frac{n}{2}}+\eps D'$ admits a factorization and we obtain
  \begin{equation*}
    D = (t-p) D' (t-\Cj{p}) - (t-p)N^{\frac{n}{2}}d_2 - d_1N^{\frac{n}{2}}(t-\Cj{p}).
  \end{equation*}
  Thus,
  \begin{multline*}
    \QNorm{D} = \Cj{D}D
    = N^2 \QNorm{D'} + N^{n+1}\QNorm{d_2} + N^{n+1}\QNorm{d_1} + N^{n}[\Cj{d}_2(t-\Cj{p})d_1(t-\Cj{p}) + (t-p)\Cj{d}_1(t-p)d_2] \\
    - N^{\frac{n+2}{2}}[(t-p)\Cj{D'}d_2 + \Cj{d}_2D'(t-\Cj{p})] - N^{\frac{n}{2}}[(t-p)\Cj{D'}(t-\Cj{p})d_1(t-\Cj{p}) + (t-p)\Cj{d}_1(t-p)D'(t-\Cj{p})].
  \end{multline*}
  By the induction hypothesis, $N^{\frac{n}{2}}$ divides $\QNorm{D'}$ whence
  $N^{\frac{n+2}{2}}$ divides the first summand. It is clear that it also divides the remaining terms.
Therefore, the statement follows.
\end{proof}

In order to extend the statement of Lemma~\ref{lem:weaklemma} to
non-translational motion polynomials, we derive some further properties on the
polynomials $Q_L$ and $Q_R$ from Lemma~\ref{lem:factorizationpattern}.

\begin{lem}\label{lem:division-property-gcd}
  Consider a bounded, monic, reduced motion polynomial $M = cQ + \eps D \in
  \D\H[t]$ of degree $n$  with $\mrpf(Q) = 1$ and $\QNorm{M} =
  c^2\QNorm{Q} = N^n$ for some irreducible, quadratic polynomial $N \in \R[t]$.
Let $M$ admit a factorization into monic linear motion
  polynomial factors and consider polynomials $Q_L$ and $Q_R$ from
  Lemma~\ref{lem:factorizationpattern}. Define $g_L \coloneqq \mygcd(c,\Cj{Q}D)$
  and $g_R \coloneqq \mygcd(c,D\Cj{Q}) \in \R[t]$. Then $g_L$ divides
  $\QNorm{Q_L}$ and $g_R$ divides $\QNorm{Q_R}$.
\end{lem}

\begin{proof}
  We only prove the lemma for $g_R$ as the statement for $g_L$ is similar. If $\deg(g_R) = 0$, i.e. $g_R = 1$, the statement is trivial, so we consider the case $0 < \deg(g_R) \leq \deg(c)$. We assume that $g_R$ does not divide $\QNorm{Q_R}$ and show that this yields to a contradiction. By Lemma~\ref{lem:factorizationpattern} we can write $M = cQ + D = (Q_L + \eps D_L) (c + \eps D') (Q_R + \eps D_R)$ with $c$ dividing $\QNorm{D'}$ by Lemma~\ref{lem:weaklemma}. Thus, we have $Q = Q_LQ_R$, $D = c(Q_LD_R + D_LQ_R) + Q_LD'Q_R$ and
  \begin{align*}
    D\Cj{Q} = D\Cj{Q}_R\Cj{Q}_L = c(Q_LD_R\Cj{Q}_R\Cj{Q}_L + D_LQ_R\Cj{Q}_R\Cj{Q}_L) + Q_LD'Q_R\Cj{Q}_R\Cj{Q}_L.
  \end{align*}
  By definition, $g_R$ divides $c$ and $D\Cj{Q}$, hence it divides
  $Q_LD'Q_R\Cj{Q}_R\Cj{Q}_L = Q_LD'\Cj{Q}_L\QNorm{Q_R}$ as well. Since $c =
  N^{\frac{m}{2}}$ for some even integer $m \in \{ 2,4,\ldots ,n \}$ we conclude
  that $g_R = N^{\frac{j}{2}}$ for some even integer $j \in \{ 2,4,\ldots ,m
  \}$. By the assumption, $g_R$ does not divide $\QNorm{Q_R}$ and therefore $N$
  divides $Q_LD'\Cj{Q}_L$. Since $N$ divides $c$ and $c$ divides $\QNorm{D'}$, we have that $N$ divides $\QNorm{D'}$. But from the decomposition
$M = cQ + D = (Q_L + \eps D_L) (c + \eps D') (Q_R + \eps D_R) = (cQ_L + \eps(cD_L + Q_LD'))(Q_R + \eps D_R)$,
we infer that $N$ does not divide 
$Q_LD'$ because otherwise
  $N$ would be a factor of $M$. We conclude that $N$ divides $D'\Cj{Q}_L$ by Lemma~\ref{lem:a-b-c-lemma}. By the definition of motion polynomial $c+\eps D'$, we notice that $\Cj{D'}=-D'$. Moreover, $N$ does not divide the conjugate of $Q_LD'$, and $\Cj{(Q_LD')}=\Cj{D'}\Cj{Q}_L = -D'\Cj{Q}_L$. This is a
  contradiction.
\end{proof}

\begin{prop}\label{prop:singlepattern}
  Consider a bounded, monic, reduced motion polynomial $M = cQ + \eps D \in
  \D\H[t]$ of degree $n$ with $\mrpf(Q) = 1$ and $\QNorm{M} =c^2\QNorm{Q}
  = N^n$ for some irreducible, quadratic polynomial $N \in \R[t]$. If $M$ admits a factorization into monic linear motion polynomial factors, then $cg_L$ or $cg_R$ divides $\QNorm{D}$, where $g_L \coloneqq
  \mygcd(c,\Cj{Q}D)$ and $g_R \coloneqq \mygcd(c,D\Cj{Q})$.
\end{prop}
\begin{proof}
  Without loss of generality we assume that $\deg(g_R) \geq \deg(g_L)$, i.e.
  $g_L$ divides $g_R$, since $M$ admits a factorization if and only if $\Cj{M}$
  does, and conjugation of $M$ interchanges $g_R$ and $g_L$. By
  Lemma~\ref{lem:factorizationpattern} we can write
  \begin{equation*}
    M = cQ + D = (Q_L + \eps D_L) (c + \eps D') (Q_R + \eps D_R) \quad \text{and} \quad D = c(Q_LD_R + D_LQ_R) + Q_LD'Q_R
  \end{equation*}
  and compute
  \begin{equation}\label{eq-E1E2}
  \begin{aligned}
    \QNorm{D} = &(c(Q_LD_R + D_LQ_R) + Q_LD'Q_R) (c(\Cj{D}_R\Cj{Q}_L + \Cj{Q}_R\Cj{D}_L) + \Cj{Q}_R\Cj{D'}\Cj{Q}_L) \\
                        = &c^2(Q_LD_R\Cj{D}_R\Cj{Q}_L + Q_LD_R\Cj{Q}_R\Cj{D}_L + D_LQ_R\Cj{D}_R\Cj{Q}_L + D_LQ_R\Cj{Q}_R\Cj{D}_L) \\
                          &+ c\underbrace{(Q_LD_R\Cj{Q}_R\Cj{D'}\Cj{Q}_L + D_LQ_R\Cj{Q}_R\Cj{D'}\Cj{Q}_L + Q_LD'Q_R\Cj{D}_R\Cj{Q}_L + Q_LD'Q_R\Cj{Q}_R\Cj{D}_L)}_{\eqqcolon E_1} \\
                          &+ \underbrace{Q_LD'Q_R\Cj{Q}_R\Cj{D'}\Cj{Q}_L}_{\eqqcolon E_2}.
  \end{aligned}
  \end{equation}
  It is sufficient to show that $g_L$ divides $E_1$ and $c g_L$ divides $E_2$.
  We have 
  \begin{equation}\label{eq-E1}
  \begin{aligned}
    E_1 &= \QNorm{Q_R}D_L\Cj{D'}\Cj{Q}_L + \QNorm{Q_R}Q_LD'\Cj{D}_L + Q_LD_R\Cj{Q}_R\Cj{D'}\Cj{Q}_L + Q_LD'Q_R\Cj{D}_R\Cj{Q}_L \\
        &= \QNorm{Q_R}D_L\Cj{D'}\Cj{Q}_L + \QNorm{Q_R}Q_LD'\Cj{D}_L + Q_L\underbrace{(D_R\Cj{Q}_R\Cj{D'} + D'Q_R\Cj{D}_R)}_{\in \R[t]}\Cj{Q}_L \\
        &= \QNorm{Q_R}D_L\Cj{D'}\Cj{Q}_L + \QNorm{Q_R}Q_LD'\Cj{D}_L + \QNorm{Q_L}(D_R\Cj{Q}_R\Cj{D'} + D'Q_R\Cj{D}_R)
  \end{aligned}
  \end{equation}
  is divisible by $g_L$ by Lemma~\ref{lem:division-property-gcd} since $g_L$
  divides $g_R$. Finally, $c$ divides $\nu(D')$ by Lemma~\ref{lem:weaklemma} and
  therefore $c g_L$ divides $E_2 = \QNorm{Q_L} \nu(D') \QNorm{Q_R}$.
\end{proof}

\begin{example}
  The factorization $M = cQ + \eps D=(Q_L + \eps D_L) (c + \eps D') (Q_R + \eps
  D_R)$ given by Lemma~\ref{lem:factorizationpattern} need not be unique. An
  example is
  \begin{align*}
(t^2+1)(t-\qi)^3+\eps \qi (t-\qi)^3 
    &= \left((t-\qi)^2+\eps \tfrac{\qj}{4}(t+\qi)\right) \left(t^2+1 +\eps (\qi - \tfrac{5}{4}\qj t + \tfrac{3}{4}\qk )\right) \left(t-\qi+\eps\qj \right) \\
    &= \left((t-\qi)+\eps \tfrac{\qj}{4}\right) \left(t^2+1 +\eps (\qi - \tfrac{5}{4}\qj t + \tfrac{3}{4}\qk )\right) \left((t-\qi)^2+\eps\qj t+\eps\qk\right).
  \end{align*}
  This motion polynomial describes the composition of a vertical Darboux
  motion, cf.~\cite[Chapter~IX, \S~3]{bottema90} or \cite{li15darboux}, with a
  rotation around the Darboux motion's axis. Motions of this type have been
  investigated by W.~Kautny in \cite{kautny56}. More precisely, our motion
  belongs to frequency $n = \frac{1}{3}$ in Kautny's terminology.
\end{example}

Since our proof of Lemma~\ref{lem:factorizationpattern} relies on a
factorization with linear factors of $M$, it is not constructive. Therefore we
are going to describe a constructive procedure to obtain $M_L=Q_L+\eps D_L$,
$M_C=c+\eps D'$, and $M_R=Q_R+\eps D_R$. The pseudo-code is given in
Algorithm~\ref{alg:singlepattern}, and the correctness of the algorithm is proved in Proposition~\ref{prop:alg_singlepattern}.

\begin{algorithm}
  \caption{\texttt{Factor3}: Factorization algorithm for motion polynomials of a primary norm}
  \label{alg:singlepattern}
  \begin{algorithmic}[1]
    \Require A bounded monic reduced motion polynomial $M=P+\eps D \in \D\H[t]$, with $\deg M = n \ge 1$, $c=
    \mrpf P$, $\deg(c)>0$, $P = cQ$, and $\QNorm{M}=N^{n}$ for a real irreducible quadratic
    polynomial $N$ such that
    $g \coloneqq \mygcd(c, \Cj{Q}D)$ divides $\mygcd(c, D\Cj{Q})$ and
    $cg$ divides~$\QNorm{D}$
    \Ensure { A triple $(M_L,M_C,M_R)$ of monic motion polynomials such that $M = M_L M_C M_R$, where $M_C = c + \eps D'$  is a product of two generic motion polynomials and $D' \in \H[t]$.}
\If{$\deg(Q)=0$}
     \State  \Return $(1, M, 1)$     
    \Else
    \If{$N$ divides $\frac{1}{g}\Cj{Q}D$} 
    \State    Pick  $q\in\H$ which does not commute with $p$, where $t-p=\lgcd(N,Q)$,
    \State $Q_L \adef \lgcd(g, Q)$, $D_L \adef qQ_L-Q_Lq$\Else 
    \State 
    $Q_L \adef \lgcd(g, Q)$, $D_L \adef 0$
    \EndIf
    \EndIf 
    \State
     ${D'_R} \adef \frac{1}{g}(c\Cj{D}_LQ + \Cj{Q}_LD)$ and $Q_R \adef \frac{\Cj{Q}_LQ}{g}$
    \State $Q_c \adef \lgcd(c, {D'_R})$, 
     $M_r \adef \Cj{Q}_cQ_R + \eps \frac{\Cj{Q}_c{D'_R}}{c}$ 
\State  Compute a factorization $M_r=L_1\cdots L_{\deg M_r}$ into monic linear factors using Algorithm~\ref{galg}\State   $M_L \adef Q_L+\eps D_L$, 
    $M_C \adef Q_c  L_1L_2\cdots L_{{\deg c}/{2}}$, and $M_R \adef L_{{\deg c}/{2}+1}\cdots L_{\deg M_r}$ \\
\Return $(M_L, M_C, M_R)$
  \end{algorithmic}
\end{algorithm}

{
\begin{prop}
  \label{prop:alg_singlepattern}
  For a  bounded, monic, reduced motion   polynomial $M=P+\eps D \in \D\H[t]$, with $\deg M = n > 1$, $c=
    \mrpf P$, $\deg(c)>0$, $P = cQ$, and $\QNorm{M}=N^{n}$ for a real, irreducible, quadratic
    polynomial $N$ such that
    $g = g_L \coloneqq \mygcd(c, \Cj{Q}D)$ divides $\mygcd(c, D\Cj{Q})$ and
    $cg$ divides $\QNorm{D}$, Algorithm~\ref{alg:singlepattern} outputs a
    triple $(M_L,M_C,M_R)$ of monic motion polynomials such that $M = M_L M_C M_R$, 
    where $M_L$ and $M_R$ are generic motion polynomials, $M_C=  c + \eps D' $ is a product of two generic motion polynomials, and $D'\in\H[t]$.
\end{prop}
\begin{proof}
  From Steps~1-3, if $\deg(Q)=0$, by the assumption, then $M= c+\eps D$ and  we have $g=1$ and $c | \QNorm{D}$. Applying Lemma~\ref{lem:g-a-b-lemma} to $A=\Cj{D}$, $B=D$, and $F=c$, we conclude that $D=M_1M_2$ with $M_1=\lgcd(c,D)$ and some $M_2\in\H[t]$ so that $c=M_1\Cj{M}_1$. Then $M=M_1\Cj{M}_1+\eps M_1M_2 =M_1(\Cj{M}_1+\eps M_2)$ is a product of two generic motion polynomials. Assume further $\deg(Q)\ne 0$.
  
  By the definition of $g$ and Lemma~\ref{lem:g-a-b-lemma}, we have $g=Q_L\Cj{Q_L}$ and $D=Q_LD_2$ for some $D_2\in\H[t]$.
Therefore, we have $g | \Cj{Q}_LD$ and $g | \Cj{Q}_LQ$, and  Step~11 outputs $D'_R$ and $Q_R$ which are polynomials. We have $M=(Q_L+\eps D_L)(cQ_R+\eps D'_R)$. Let us check that the second factor still satisfies the assumptions of the proposition. We make a case distinction for the divisibility of $\frac{1}{g}\Cj{Q}D$ by $N$.
  
  Case I. $N$ divides $\frac{1}{g}\Cj{Q}D$. Then we can always pick $q\in\H$ with $pq\neq qp$, where $t-p=\lgcd(N, Q)$, and  we have $t-p=\lgcd(N, D)$  by Lemma~\ref{lem:g-a-b-lemma}. In addition, we have $\deg g = \deg c$, i.e., $g = c$, and the following properties for $D'_R$:
 \begin{itemize}
\item We claim that $N$ does not
  divide $\Cj{Q}_R{D'_R} =(c\Cj{Q}_R\Cj{D}_LQ + \Cj{Q}_R\Cj{Q}_LD)/g$ with our choice of $D_L$.  Since $N$ divides $\Cj{Q}D/g=\Cj{Q}_R\Cj{Q}_LD/g$, the claim is that $N$ is not a factor of $\Cj{Q}_R\Cj{D}_LQ$. Otherwise, $N$ divides $\Cj{Q}_R\Cj{D}_LQ$, which implies that $N$ divides $\Cj{Q}_R\Cj{Q}_L\Cj{q}Q$ by the definition of $D_L$ and $\deg g>0$.
    Since $t-p=\lgcd(N,Q)$, it follows that $N$ divides $(t-\Cj{p})\Cj{q}Q$ by 
    Lemma~\ref{lem:a-b-c-lemma} applied to $A=\tfrac{\Cj{Q}(t-p)}{N}$, $B=t-\Cj{p}$, and $C=\Cj{q}Q$. In addition, $N$ divides $(t-\Cj{p})\Cj{q}(t-p)$, again by 
    Lemma~\ref{lem:a-b-c-lemma} applied to $A=(t-\Cj{p})$, $B=\Cj{q}(t-p)$, and $C=(t-\Cj{p})Q/N$, which contradicts the choice of $q$.
\item We claim that $c | D'_R\Cj{D'_R}$. Since $c=g$, we have $c | Q\Cj{Q} | \Cj{D}_LQ\Cj{Q}D_L$ and $c| Q_L \Cj{Q}_L\frac{D\Cj{D}}{g^2}$. 
Since $D=Q_LD_2$ and $g^2 | \Cj{D}D$, it follows that $g | D_2\Cj{D}_2$. 
Since $M$ is reduced, it follows that $\mrpf(g,D_2)=\mrpf(g,D)=1$, and $g=\QNorm{\rgcd(D_2,g)}=\QNorm{\rgcd(D,g)}$  by Lemma~\ref{lem:g-a-b-lemma}, hence $\rgcd(D_2,g)=\rgcd(D,g)$. Since $g | D\Cj{Q}$, it follows that $g=\rgcd(g,D)\lgcd(g,\Cj{Q})=\rgcd(g,D_2)\lgcd(g,\Cj{Q})$  by Lemma~\ref{lem:g-a-b-lemma}, thus $g | D_2\Cj{Q}$.
Therefore, $g^2$ divides $Q\Cj{D}Q_L$ and $\Cj{Q}_L{D}\Cj{Q}$. Therefore, $c$ divides $D'_R\Cj{D'_R}$ because we have
    $D'_R\Cj{D'_R}=\Cj{D}_LQ\Cj{Q}D_L + Q_L \Cj{Q}_L\frac{D\Cj{D}}{g^2} +\frac{c}{g^2} \left(\Cj{D}_LQ\Cj{D}Q_L + \Cj{Q}_L{D}\Cj{Q}D_L \right)$.
 \end{itemize} 
 
 Case II. $N$ does not divide $\frac{1}{g}\Cj{Q}D$. The real polynomial $c$ divides $D'_R\Cj{D'_R}$ again because $cg$ divides $D\Cj{D}$, and $N$ does not
  divide $\Cj{Q}_RD'_R=\frac{1}{g}\Cj{Q}D$ by the assumption.

 Therefore, in both cases,
 $c$ divides $D'_R\Cj{D'_R}$, hence by Lemma~\ref{lem:g-a-b-lemma} we have $c= Q_c\Cj{Q}_c$, and $c$ divides $\Cj{Q}_cD'_R$. 
 Therefore, the outputs of Step~11 are right we need, and 
 Step~12 outputs the polynomial $M_r=\Cj{Q}_cQ_R + \eps \frac{\Cj{Q}_cD'_R}{c}$ which is a generic motion polynomial. 

Therefore,  
Step~13 gives a factorization of $M_r$, hence its primal part $\Cj{Q}_cQ_R$, into monic linear factors. Such a factorization is unique because $\nu (M_r)$ is a power of $N$, hence the linear factors are reconstructed one by one by division by $N$ with a remainder. Hence by Lemma~\ref{lem:g-a-b-lemma} applied to $F=c$, the primal part of $L_1L_2\cdots L_{{\deg c}/{2}}$ equals $\lgcd(c,\Cj{Q}_cQ_R)=\Cj{Q}_c\lgcd({Q}_c,Q_R)=\Cj{Q}_c\lgcd(c,D'_R,Q_R)=\Cj{Q}_c$ because $c$ divides $N^n$ and $N$ does not divide $\Cj{Q}_RD'_R$. Hence the primal part of $M_C$ is $c$. Step~14 gives two generic motion polynomials $M_L$ and $M_R$. By a straightforward calculation, $M=(Q_L+\eps D_L)(cQ_R+\eps D'_R)=M_LQ_cM_r=M_LM_CM_R$.

\end{proof}
}

{
In summary, we gave a constructive procedure in Algorithm~\ref{alg:singlepattern} to find the factorization of the
motion polynomial $M=cQ+\eps D$ when $cg$ divides $\nu(D)$ and $\QNorm{M}=N^{n}$ for a real, irreducible, quadratic
    polynomial $N$. 
}

\begin{prop}\label{prop:reverse-singlepattern}
  Consider a bounded, monic, reduced motion polynomial $M = cQ + \eps D \in
  \D\H[t]$ of degree $n \in \N$ with $\mrpf(Q) = 1$ and $\QNorm{M} =c^2\QNorm{Q}
  = N^n$ for some irreducible, quadratic polynomial $N \in \R[t]$. Set $g
  \coloneqq \mygcd(c,\Cj{Q}D,D\Cj{Q}) \in \R[t]$. If $cg$ divides $\QNorm{D}$,
  then $M$ admits a factorization into monic linear factors.
\end{prop}

\begin{proof}
  We may assume that $\mygcd(c,\Cj{Q}D)$ divides $\mygcd(c,D\Cj{Q})$; otherwise, replace $M$ with $\Cj{M}$. Apply Proposition~\ref{prop:alg_singlepattern}
to obtain $M_L$, $M_C$, and $M_R$ such that $M = M_LM_CM_R$. The polynomials $M_L$ and $M_R$ are generic and can be decomposed into monic linear factors by Algorithm~\ref{galg}. The polynomial $M_C$ is translational but, as product of two generic motion polynomials, also admits a factorization with monic linear motion polynomial factors.
\end{proof}

\subsection{Factorizations of General Bounded Motion Polynomials}
\label{sec:general}

The final step is to extend the result of Proposition~\ref{prop:singlepattern} to
motion polynomials with arbitrary norms. We combine
Proposition~\ref{prop:alg_mgfactor}, Lemma~\ref{lem:factorizationpattern}, and a
technique called ``Bennett flips'' introduced in \cite{li17}. It is based on
replacing a factorization $(t-h_1)(t-h_2)$ of a generic 
monic motion
polynomial of degree two with $\QNorm{t-h_1}\ne \QNorm{t-h_2}$ by its second factorization $(t-k_1)(t-k_2)$ \cite[Theorem 1]{hegedus13}. Note that
$\QNorm{t-h_1} = \QNorm{t-k_2}$ and $\QNorm{t-h_2} = \QNorm{t-k_1}$. First, we
generalize the latter procedure.

\begin{lem}\label{lem:bennett-flips}
  Consider a bounded monic reduced motion polynomial $M = c + \eps D \in
  \D\H[t]$ of degree $n$ with $c \in \R[t]$ and a linear bounded monic
  reduced motion polynomial $t - p - \eps d \in \D\H[t]$  with $p,d\in\H$. Let $\QNorm{M}
  = c^2 = N_1^n$ and $\nu(t - p - \eps d) = N_2$ for some irreducible, quadratic
  polynomials $N_1$, $N_2 \in \R[t]$ with $\mygcd(N_1,N_2) = 1$. If $M$ admits a
  factorization into monic linear motion polynomial factors, then there exists a quaternion polynomial $D_1 \in \H[t]$ and
  a quaternion $d_1 \in \H$ such that
$M (t - p - \eps d) = (t - p - \eps d_1)(c + \eps D_1)$, and $c + \eps D_1$ admits a factorization into monic linear motion polynomial factors.
\end{lem}
\begin{proof}
  Writing a factorization of $M$ as $M = (t - h_1)(t - h_2)\cdots(t - h_n)$ and using one Bennett flip, we have
$ M(t - p - \eps d) = (t - h_1)(t - h_2) \cdots (t - h_{n-1})(t - p_n - \eps d_n)(t - \ell_n)$,
where $(t - h_n)(t - p - \eps d) = (t - p_n - \eps d_n)(t - \ell_n)$, $\nu(t - h_n) = \nu(t - \ell_n) = N_1$, and $\nu(t - p_n - \eps d_n) = \nu(t - p - \eps d) = N_2$. Applying this procedure $n$ times, we arrive at
$ M(t - p - \eps d) = (t - p_1 - \eps d_1)(t - \ell_1)(t - \ell_2) \cdots (t - \ell_n)$
with $\nu(t - p_1 - \eps d_1) = N_2$ and $\nu((t - \ell_1)(t - \ell_2) \cdots (t - \ell_n)) = N_1^n$. Comparing the primal parts and applying Lemma~\ref{lem:a-b-c-lemma} to $A=t-\Cj{p_1}$, $B=t-p$, $C=c$, and $N=N_{2}$, we conclude $p_1 = p$ and $(t - \ell_1)(t - \ell_2) \cdots (t - \ell_n) = c + \eps D_1$ for some $D_1 \in \H[t]$.
\end{proof}

\begin{lem}\label{lem:m_r_admit_a_factorization}
    Let $M = \prod_{i=1}^r M_i$ be a factorization of a bounded monic reduced motion polynomial $M$ given by Proposition~\ref{prop:alg_mgfactor}. If $M$ admits a factorization into monic linear motion polynomial factors, then each $M_i$ does. \end{lem}
\begin{proof}
  We prove the statement by induction on $r$. 
  The base $r = 1$ is trivial. Take $r>1$ and consider any factorization $M= \prod_{i=1}^r M_i$ of $M$ given by
  Proposition~\ref{prop:alg_mgfactor} with $\QNorm{M_1}= N_1^{n_1}$, \dots, $\QNorm{M_r}= N_r^{n_r}$ for some pairwise coprime real irreducible quadratic polynomials $N_1$,     $N_2$, \ldots, $N_r$ and positive integers $n_1$, $n_2$, \ldots, $n_r$. Using Bennett flips, we get $M= \prod_{i=1}^r \prod_{j=1}^{n_i} L_{ij}$, where $L_{ij}$ are monic linear motion polynomial factors with $\nu(L_{ij})=N_i$. Let $\tilde{M}_r=\prod_{j=1}^{n_r} L_{rj}$. Then we have $M\Cj{\tilde{M}_r}=N_r^{n_r}\prod_{i=1}^{r-1} \prod_{j=1}^{n_i} L_{ij}=:N_r^{n_r}\tilde{M}'$.   Since $\nu(\tilde{M}')=\prod_{i=1}^{r-1} N_i^{n_i}$, by applying the induction hypothesis to $\tilde{M}'$,
we only need to prove that $\tilde{M}_r=M_r$.  It is clear that $ N_r^{n_r}$ divides $M\Cj{\tilde{M}_r}$. First, $ N_r^{n_r}$ divides the primal part of $M\Cj{\tilde{M}_r}=(P_1+\eps D_1)(P_2+\eps D_2)=P_1P_2+\eps(P_1D_2+D_1P_2)$ with $\prod_{i=1}^{r-1} M_{i}=P_1+\eps D_1$ and $M_r\Cj{\tilde{M}_r}=P_2+\eps D_2$. This implies that $ N_r^{n_r}$ divides $P_1P_2$, then $ N_r^{n_r}$ divides $P_2$ by Lemma~\ref{lem:g-a-b-lemma}. Second, $ N_r^{n_r}$ divides the dual part $P_1D_2+D_1P_2$.  Then $ N_r^{n_r}$ also divides $P_1D_2$. Then $ N_r^{n_r}$ divides $D_2$ by Lemma~\ref{lem:g-a-b-lemma}. Then $ N_r^{n_r}$ divides  $M_r\Cj{\tilde{M}_r}$, and they must be equal. Then $M_r=\tilde{M}_r$, which finishes the proof.
\end{proof}

\begin{lem}\label{lem:multiplicativity-gcd}
  Consider a bounded, monic, reduced motion polynomial $M = c Q + \eps D \in \DH[t]$ with $c \in \R[t]$ and $\mrpf(Q) = 1$. Let $M = \prod_{i=1}^r (c_iQ_i + \eps D_i)$ be a factorization of $M$ given by Proposition~\ref{prop:alg_mgfactor} with $c_i \in \R[t]$ and $\mrpf(Q_i) = 1$. 
  For $i = 1,\ldots ,r$, set 
  \begin{align*}
  g_{iL}&\coloneqq\mygcd(c_i,\Cj{Q}_iD_i), 
  &  g_{iR}&\coloneqq\mygcd(c_i,D_i\Cj{Q}_i), 
  &
  g_i&\coloneqq \mygcd(c_i,D_i\Cj{Q}_i,\Cj{Q}_iD_i),\\ 
g_L &\coloneqq \mygcd(c,\Cj{Q}D), 
  &
  g_R &\coloneqq \mygcd(c,D\Cj{Q}), 
  & 
  g
  &\coloneqq \mygcd(c,D\Cj{Q},\Cj{Q}D).
  \end{align*}
  Then $g_L = \prod_{i=1}^r g_{iL}$, $g_R = \prod_{i=1}^r g_{iR}$,  and $g = \prod_{i=1}^r g_{i}$.

\end{lem}
\begin{proof}
  We prove the statement by induction on $r$. The base $r = 1$ is trivial. Take $r>1$ and consider the factorization $M
  = cQ + \eps D = \prod_{i=1}^{{r}} (c_iQ_i + \eps D_i)$ of $M$ given by
  Proposition~\ref{prop:alg_mgfactor}. We have
$Q = \prod_{i=1}^{{r}} Q_{i}, \quad D = \sum_{i=1}^{{r}} \left[ \left( \prod_{j=1}^{i-1} c_jQ_j \right) D_i \left( \prod_{j=i+1}^{{r}} c_jQ_j \right) \right]$,
and therefore
  \begin{equation*}
    \begin{aligned}
      D\Cj{Q} &= \sum_{i=1}^{{r-1}} \left[ \left( \prod_{j=1}^{i-1} c_jQ_j \right) D_i \left( \prod_{j=i+1}^{{r}} c_jQ_j \right) \right] \left( \prod_{i=1}^{{r}} \Cj{Q}_{r+{1}-i} \right) + \left( \prod_{j=1}^{r-1} c_jQ_j \right) D_{{r}} \left( \prod_{i={1}}^{r} \Cj{Q}_{{r+1-i}} \right) \\
              &= \sum_{i=1}^{{r-1}} \left[ \left( \prod_{j=1}^{i-1} c_jQ_j \right) D_i \left( \prod_{j=i+1}^{{r-1}} c_jQ_j \right) \right] \left( \prod_{i=1}^{{r-1}} \Cj{Q}_{{r-i}} \right) c_{{r}}\nu(Q_{{r}}) \\
              &\qquad + \left( \prod_{j=1}^{r-1} c_j \right) \left( \prod_{j=1}^{r-1} Q_j \right) D_{{r}}\Cj{Q}_{{r}} \left( \prod_{i=1}^{{r-1}} \Cj{Q}_{{r-i}} \right) \\
              &= \tilde{D} \Cj{\tilde{Q}} c_{{r}}\nu(Q_{{r}}) + \tilde{c} \tilde{Q} D_{{r}}\Cj{Q}_{{r}} \Cj{\tilde{Q}},
    \end{aligned}
  \end{equation*}
  where $\tilde{c} \tilde{Q} + \eps \tilde{D} \coloneqq \prod_{i=1}^{r-1} (c_iQ_i +
  \eps D_i)$. Basic $\mygcd$-properties and the definitions of $c_{i}$ and $Q_i$
  yield
  \begin{align*}
    \mygcd(\tilde{c}, \tilde{D} \Cj{\tilde{Q}}) & \overset{\mygcd(\tilde{c},c_{{r}}\nu(Q_{{r}}))=1}{=} \mygcd(\tilde{c}, \tilde{D}\Cj{\tilde{Q}} c_{{r}} \nu(Q_{{r}})) \\
    &= \mygcd(\tilde{c}, \tilde{D}\Cj{\tilde{Q}} c_{{r}} \nu(Q_{{r}}) + \tilde{c} \tilde{Q} D_{{r}}\Cj{Q}_{{r}} \Cj{\tilde{Q}}) = \mygcd(\tilde{c},D\Cj{Q}), \\
    \mygcd(c_{{r}}, D_{{r}}\Cj{Q}_{{r}}) & \overset{\mygcd(c_{{r}},{\tilde{c}\nu(\tilde{Q})})=1}{=} \mygcd(c_{{r}}, \tilde{c}\tilde{Q}D_{{r}}\Cj{Q}_{{r}}\Cj{\tilde{Q}}) \\
     & = \mygcd(c_{{r}}, \tilde{D}\Cj{\tilde{Q}} c_{{r}} \nu(Q_{{r}}) + \tilde{c} \tilde{Q} D_{{r}}\Cj{Q}_{{r}} \Cj{\tilde{Q}}) = \mygcd(c_{{r}}, D\Cj{Q}).
  \end{align*}
  We apply the induction hypothesis to $\tilde{c}\tilde{Q} + \eps \tilde{D}$,
  i.e. $\mygcd(\tilde{c}, \tilde{D}\Cj{\tilde{Q}}) = \prod_{i=1}^{r-1} g_{iR}$, and
  obtain
  \begin{align*}
    \prod_{i=1}^{{r}} g_{iR} &= g_{{r}R} \prod_{i=1}^{r-1} g_{iR} = \mygcd(c_{{r}}, D_{{r}}\Cj{Q}_{{r}}) \mygcd(\tilde{c}, \tilde{D}\Cj{\tilde{Q}}) \\
                             &= \mygcd(c_{{r}}, D\Cj{Q}) \mygcd(\tilde{c}, D\Cj{Q}) \overset{\mygcd(c_{{r}},\tilde{c})=1}{=} \mygcd(c,D\Cj{Q}) = g_R.
  \end{align*}
  The statement on $g_L$ is proved analogously.
  
  Finally, since $\mygcd(g_{iL}, g_{jR}) = 1$ for $i \neq j$, it follows that
  \begin{equation*}
    g = \mygcd(g_L,g_R) = \mygcd\left( \prod_{i=1}^r g_{iL}, \prod_{i=1}^r g_{iR} \right) = \prod_{i=1}^r \mygcd(g_{iL}, g_{iR}) = \prod_{i=1}^r g_i.
  \end{equation*}
\end{proof}

\begin{prop}\label{prop:final-characterization}
  Consider a bounded, monic, reduced motion polynomial $M = c Q + \eps D \in
  \DH[t]$ with $c \in \R[t]$ and $\mrpf(Q) = 1$. Define $g \coloneqq
  \mygcd(c,D\Cj{Q},\Cj{Q}D)$. If $M$ admits a factorization into monic linear factors,
  then $c g$ divides $\QNorm{D}$.
\end{prop}

\begin{proof}
  Let $M = \prod_{i=1}^r (c_iQ_i + \eps D_i)$ be a factorization of $M$ given by Proposition~\ref{prop:alg_mgfactor}. By Lemma~\ref{lem:m_r_admit_a_factorization}, all the factors admit further factorization into monic linear motion polynomials. We apply Lemma~\ref{lem:factorizationpattern} to each factor of $M$ and obtain
  \begin{equation*}
    M = \prod_{i=1}^r(c_iQ_i + \eps D_i) = \prod_{i=1}^r (Q_{iL} + \eps D_{iL}) (c_i + \eps D_i') (Q_{iR} + \eps D_{iR})
  \end{equation*}
  for some quaternionic polynomials $Q_{iL},D_{iL},D'_i,Q_{iR},D_{iR}$ such that $Q_{iL} + \eps D_{iL}$, $c_i + \eps D_i'$, and $Q_{iR} + \eps D_{iR}$ admit factorizations into monic linear factors. Then we can use Bennett flips from Lemma~\ref{lem:bennett-flips} to rearrange these factors in the following way: First we flip the factor $Q_{2L} + \eps D_{2L}$ over to the right of the factor $Q_{1L} + \eps D_{1L}$ and denote the result by $\tilde{Q}_{2L} + \eps \tilde{D}_{2L}$. Next we flip $Q_{3L} + \eps D_{3L}$ over to the right of the flipped factor $\tilde{Q}_{2L} + \eps \tilde{D}_{2L}$ and continue in this way for $i = 4,\ldots ,r$. In the same manner we flip the factors $Q_{iR} + \eps D_{iR}$ to the left of $Q_{1R} + \eps D_{1R}$. This procedure yields
  \begin{equation}\label{eq:final-characterization}
    \begin{aligned}
      M &= \prod_{i=1}^r (Q_{iL} + \eps D_{iL}) (c_i + \eps D_i') (Q_{iR} + \eps D_{iR}) \\
        &= \left( \prod_{i=1}^r (\tilde{Q}_{iL} + \eps \tilde{D}_{iL}) \right) \left(
          \prod_{i=1}^r (c_i + \eps \tilde{D}_i') \right) \left( \prod_{i=1}^r
          (\tilde{Q}_{iR} + \eps \tilde{D}_{iR}) \right) \\
          & \eqqcolon (Q_L + \eps D_L) (c
          + \eps D') (Q_R + \eps D_R),
    \end{aligned}
  \end{equation}
  for some $Q_{L},D_{L},D',Q_{R},D_{R}\in\H[t]$, where $c = \prod_{i=1}^r c_i$, $\QNorm{Q_L} = \prod_{i=1}^r \nu(\tilde{Q}_{iL}) = \prod_{i=1}^r \nu(Q_{iL})$, and $\QNorm{Q_R} = \prod_{i=1}^r \nu(\tilde{Q}_{iR}) = \prod_{i=1}^r \nu(Q_{iR}) $ by Lemma~\ref{lem:bennett-flips}. 
  Use the notation from Lemma~\ref{lem:multiplicativity-gcd}.
  Lemmas~\ref{lem:division-property-gcd} and~\ref{lem:multiplicativity-gcd} imply that $\prod_{i=1}^r g_{iL}$ divides $\QNorm{Q_L}$, and $\prod_{i=1}^r g_{iR}$ divides $\QNorm{Q_R}$. 
Consequently, $g$ divides both $\QNorm{Q_L}$ and $\QNorm{Q_R}$. 
  
  It remains to prove that $E_1$ and $E_2$ given by Eq.~\eqref{eq-E1E2} above are divisible by $g$ and $cg$ respectively. Here $g|E_1$ by Eq.~\eqref{eq-E1}.
From
  Eq.~\eqref{eq:final-characterization} we infer that
  \begin{equation*}
    D' = \sum_{i=1}^r {\tilde{D}_i'}\prod_{\substack{j=1\\ j \neq i}}^r c_j  \quad \text{and} \quad
    \nu(D') = \left( \sum_{i=1}^r {\tilde{D}_i'}\prod_{\substack{j=1\\ j \neq i}}^r c_j \right) \left( \sum_{k=1}^r {\tilde{D}_k^{\prime\ast}}\prod_{\substack{j=1\\ j \neq k}}^r c_j \right).
  \end{equation*}
  We see that $\nu(D')$ is the sum of expressions like
  \begin{equation*}
    \left( {\tilde{D}_i'}\prod_{\substack{j=1\\ j \neq i}}^r c_j \right) \left( {\tilde{D}_k^{\prime\ast}} \prod_{\substack{j=1\\ j \neq k}}^r c_j \right) =
    \begin{cases}
      {\tilde{D}_i' {\tilde{D}_k^{\prime\ast}}c_i c_k}
      \displaystyle\prod_{\substack{j=1\\ j \neq i,k}}^r c_j^2 & \text{if $i \neq k$,} \\
      {\nu(\tilde{D}_i')}
      \displaystyle\prod_{\substack{j=1\\ j \neq i}}^r c_j^2 & \text{if $i = k$}
    \end{cases}
  \end{equation*}
  with $i,k \in \{ 1, 2, \ldots ,r \}$. The polynomials $c_i + \eps
  \tilde{D}_i'$ admit a factorization into monic linear factors. Therefore, $c_i$
  divides $\nu(\tilde{D}_i')$ for $i = 1,2,\ldots ,r$ by Lemma~\ref{lem:weaklemma}
  and each summand of $\nu(D')$, in particular, $\nu(D')$ itself, is divisible by
  $c= \prod_{i=1}^r c_i$. Thus $c g\,|\,E_2 = \QNorm{Q_L} \nu(D')
  \QNorm{Q_R}$.
\end{proof}

Note that Proposition~\ref{prop:final-characterization} agrees with
Proposition~\ref{prop:singlepattern} in the following way. If the norm polynomial of
$M$ has only one irreducible quadratic factor $\QNorm{M} = N^n$, then $g_L =
N^{n_L}$ and $g_R = N^{n_R}$ for some integer $n_L$, $n_R$ and $g =
\mygcd(g_L,g_R)$ is either equal to $g_L$ or $g_R$.

\begin{prop}\label{prop:reverse-final}
  Consider a bounded, monic, reduced motion polynomial $M = c Q + \eps D \in
  \DH[t]$ with $c \in \R[t]$ and $\mrpf(Q) = 1$. Set $g \coloneqq
  \mygcd(c,\Cj{Q}D,D\Cj{Q}) \in \R[t]$. If $cg$ divides $\QNorm{D}$, then $M$ admits
  a factorization into monic linear factors.
\end{prop}

\begin{proof} Use the notation from Lemma~\ref{lem:multiplicativity-gcd}. It suffices to prove that $c_jg_j$ divides $\QNorm{D_j}$ for each $j=1,\dots,r$; then by Proposition~\ref{prop:reverse-singlepattern}, each factor $M_j$, hence $M$ itself,
  admits a factorization into monic linear factors.

  By expanding the product $\prod_{i=1}^r M_i$, we
  have
  \begin{equation*}
    \begin{aligned}
      D &= \sum_{i=1}^r \left( \prod_{k=1}^{i-1} c_kQ_k \right) D_i \left( \prod_{k=i+1}^r c_kQ_k \right) \\
        &= \underbrace{\left(\prod_{i=1}^{j-1} c_iQ_i \right)D_j\left(\prod_{i=j+1}^rc_iQ_i\right)}_{\coloneqq A} + \underbrace{\left(\prod_{i=1}^{j-1}c_iQ_i\right)c_jQ_jX}_{\coloneqq B} + \underbrace{Yc_jQ_j\left(\prod_{i=j+1}^rc_iQ_i\right)}_{\coloneqq C},
    \end{aligned}
  \end{equation*}
  for some quaternion polynomials $X$ and $Y$. Now
  \begin{align*}
    \QNorm{D} &=
    A\Cj{A} + B\Cj{B} + C\Cj{C} + (A\Cj{B} + B\Cj{A}) + A\Cj{C} + C\Cj{A} +
    B\Cj{C} + C\Cj{B}\\
              &=
     A\Cj{A} + B\Cj{B} + C\Cj{C} + (\Cj{B}A + \Cj{A}B) + A\Cj{C} + C\Cj{A} +
    B\Cj{C} + C\Cj{B}.
  \end{align*}
  Clearly, $B\Cj{B}$, $C\Cj{C}$, $B\Cj{C}$, and $C\Cj{B}$ are divisible by
  $c_j^2$ and hence also by $c_jg_j$. Moreover, $A\Cj{C}$ and $C\Cj{A}$ are
  divisible by $c_j\mygcd(D_j\Cj{Q_j})$, hence by $c_jg_j$. Analogously, $\Cj{B}A$ and $\Cj{A}B$ are divisible by $c_jg_j$. If $cg$ divides $\QNorm{D}$, then $c_jg_j$ divides $\QNorm{D}$ by Lemma~\ref{lem:multiplicativity-gcd}. Thus $A\Cj{A}=\QNorm{D_j}\prod_{\substack{i=1\\ i \neq j}}^r c_i^2 \QNorm{Q_i}$ is divisible by $c_jg_j$. Since the product here is coprime with $c_jg_j$, it follows that  $c_jg_j$ divides $\QNorm{D_j}$, as required.
\end{proof}

\begin{proof}[Proof of Theorem~\ref{th:bounded}]
    This follows directly from Propositions~\ref{prop:final-characterization}--\ref{prop:reverse-final}.
\end{proof}

\subsection{A Comprehensive Example}
\label{sec:example}

Notice that Algorithm~\ref{alg:singlepattern} for finding factorizations with
monic linear factors is different from Algorithm~3 in \cite{li19jsc}.

Let us illustrate the superiority of Algorithm~\ref{alg:singlepattern} over
\cite[Algorithm~3]{li19jsc} by a comprehensive example. We wish to factor the
motion polynomial $M = P + \eps D$ where
\begin{equation}
  \label{eq:7}
    P = (t^2+1)(t-\qi)^2,\quad
    D = \qi(t-\qi)^2.
\end{equation}
It is a Kautny motion of frequency $n = \frac{1}{2}$; cf.~\cite{kautny56}.

Algorithm~3 of \cite{li19jsc} yields a factorization of
$(t^2+1)M$ into a product of six linear motion polynomials. As we shall see, our
Algorithm~\ref{alg:singlepattern} is capable of decomposing $M$ directly,
without the need to multiply with a real polynomial, into four linear motion
polynomial factors. Note that the presence of the right factor $(t-\qi)^2$ is
not helpful in either algorithm.

We follow the steps of Algorithm~\ref{alg:singlepattern}:
\begin{itemize}
\item Since $\QNorm{M}=(t^2+1)^4$, $c=\mrpf P = t^2+1$, $Q=(t-\qi)^2$, $D=\qi(t-\qi)^2$, $g=
  \mygcd(c,\Cj{Q}D,\Cj{D}Q)=t^2+1$, and $cg$ divides $\nu(D)$, we can apply
  Algorithm~\ref{alg:singlepattern}.
\item We get $Q_L=t-\qi$ and $D_L=\qj a+\qk b$ for two
  arbitrary real numbers $(a, b)\ne (0,0)$. \item We get $D'_R=\frac{c\Cj{D}_LQ}{g}+\frac{\Cj{Q}_LD}{g}=-(a \qj + b\qk) t^2 + (\qi -
  2a\qk + 2 b\qj) t + a\qj + b\qk + 1$ and $Q_R=t-\qi$.
\item We find
$Q_c = \lgcd(c,D'_R) =
    t + \frac{1}{4a^2+4b^2+1}((4a^2+4b^2-1)\qi - 4b\qj + 4a\qk)$
and 
  \begin{equation*}
    \begin{aligned}
        M_r=\Cj{Q_c}Q_R + \eps \frac{\Cj{Q}_cD'_R}{c} 
&=t^2 - \frac{ 8 (a^2 + b^2) \qi + 4 a \qk - 4 b \qj} {4 a^2 + 4 b^2 + 1}t -1+ \frac{ 4 a \qj + 4 b \qk +2}{4 a^2 + 4 b^2 + 1} \\
            &\quad  -\eps(a\qj+b\qk)t
            + \eps\frac{\qi + (4a^2+4b^2+3)(b\qj - a\qk)}{4a^2+4b^2+1}. \\
\end{aligned}
  \end{equation*}
\item By Algorithm~\ref{galg}, $M_r=\left(\Cj{Q_c}-\eps\tfrac{(4a^2+4b^2+1)(a\qj + b\qk)}{4 (a^2 + b^2)}\right)\left(Q_R + \eps \frac{a\qj +b\qk}{4 (a^2 + b^2)}\right)$.
\item Finally, we have $M_L=Q_L+\eps D_L$, $M_C=Q_c 
         \left(\Cj{Q}_c-\eps\tfrac{(4a^2+4b^2+1)(a\qj + b\qk)}{4 (a^2 + b^2)}\right)$, and $M_R=t-\qi+\eps\frac{a\qj +b\qk}{4 (a^2 + b^2)}$, and 
the desired factorization
      \begin{multline*}
          (t^2+1)(t-\qi)^2+\eps\qi(t-\qi)^2=\left(t-\qi+\eps(\qj a+\qk b)\right)
          \left( t + \tfrac{(4a^2+4b^2-1)\qi - 4b\qj + 4a\qk}{4a^2+4b^2+1}\right)\\
          \left( t - \tfrac{(4a^2+4b^2-1)\qi - 4b\qj + 4a\qk}{4a^2+4b^2+1}-\eps\tfrac{(4a^2+4b^2+1)(a\qj + b\qk)}{4 (a^2 + b^2)}\right)
          \left(t-\qi+\eps\tfrac{a\qj +b\qk}{4 (a^2 + b^2)}\right).
      \end{multline*}
      For instance, for $a=1$ and $b=0$, we get
      \begin{multline*}
          (t^2+1)(t-\qi)^2+\eps\qi(t-\qi)^2=\left(t-\qi+\eps\qj\right)
          \left( t + \tfrac{3\qi+ 4\qk}{5}\right)
          \left( t - \tfrac{3\qi+4\qk}{5}-\eps\tfrac{5\qj}{4}\right)
          \left(t-\qi+\eps\tfrac{\qj}{4}\right).
      \end{multline*}
\end{itemize}

{

\section{Another proof of Theorem~\ref{th:bounded}}\label{ap:simpler}

Now we give an alternative proof of the main result by induction and an alternative factorization algorithm. This section is independent of the previous one, besides applying known Theorem~\ref{thm:alg_galg} stated there. Throughout this section, for nonzero $X\in\H[t]$ and $N\in\R[t]$, denote
 by $\tau(X)$ the maximal integer $\tau$ such that~$N^\tau|X$.

\begin{proof}[Proof of Theorem~\ref{th:bounded}]
 [The `if' part]
 We give proof by induction on the degree of $P:=cQ$. If $c=1$, then just apply known Theorem~\ref{thm:alg_galg}. It suffices to prove that for
 $c\neq 1$ we can write $M=(t-p-\eps q)M_1$ or $M=M_1(t-p-\eps q)$ for some
$p, q\in\H$ such that  $q$ is vectorial, $\Cj{p}q=qp$ and some motion
 polynomial $M_1$ still satisfying the assumptions of the theorem.

 Pick a monic real irreducible divisor $N$ of $c$. Without loss of
 generality, assume $\tau(D\Cj{P})\geq \tau(\Cj{P}D)$, otherwise replace $M$ by
 $\Cj{M}$.

 Since $N|c$ and $c|\Cj{D}D$,  by Lemma~\ref{lem:g-a-b-lemma} it follows that $D=(t-p)D_1$ for some $p\in\H$ and
 $D_1 \in \H[t]$, where $\nu(t-p)=N$. Moreover, we have $P=(t-p)P_1$ for some
 $P_1\in\H[t]$. Set $M_1=P_1+\eps(D_1+qP/N)$, where $q\in\H$ is going to be specified below, so that $M=(t-p-\eps q)M_1$.
 Define $c_1$ and $g_1$ for $M_1$ analogously to $c$ and $g$.

 Case 1: $\tau(P_1)<\tau(P)$ or $\tau(\Cj{P}D)\le2\tau(P)$. In this case, we set
 $q=0$. We get
  \begin{gather*}
   P =(t-p)P_1, \quad \Cj{P}D=\Cj{P_1}(t-\Cj{p})(t-p)D_1=N\Cj{P_1}D_1, \quad D\Cj{P}=(t-p)D_1\Cj{P_1}(t-\Cj{p}), \\
   cg=\mygcd\left(c^2,\Cj{P}D,D\Cj{P}\right), \quad
   c_1g_1 = \mygcd\left(c_1^2, \Cj{P_1}D_1,D_1\Cj{P_1}\right),
  \end{gather*}
where $c_1g_1|cg|\nu(D)$ and $\nu(D)=N\nu(D_1)$. It remains to prove that $\tau(c_1g_1)<\tau(cg)$. By the assumption $\tau(D\Cj{P})\geq \tau(\Cj{P}D)=\tau(\Cj{P_1}D_1)+1$, we get
\begin{align*}
 \tau(cg)& =\min\{2\tau(P),\tau(\Cj{P}D),\tau(D\Cj{P})\}  \\
 & =\min\{2\tau(P),\tau(\Cj{P_1}D_1)+1\}>\min\{2\tau(P_1),\tau(\Cj{P_1}D_1)\}\geq\tau(c_1g_1)
\end{align*}
because one of the inequalities $\tau(P_1)<\tau(P)$ or $\tau(\Cj{P}D)\leq
2\tau(P)$ holds. We get $c_1g_1| D_1\Cj{D_1}$, hence  
$M_1$ satisfies the assumptions of the theorem and $M=(t-p)M_1$.

Case 2: $\tau(P_1)=\tau(P)$ and $\tau(\Cj{P}D) \geq 2\tau(P)+1$. In this case,
  we set $q$ to be any nonzero vectorial quaternion such that $\Cj{p}q=qp$. In other words, $q=pv-vp$ for some vectorial quaternion $v$ that does not commute with $p$. 
  
  This case is reduced to the previous one by the transformation $M\mapsto (1+\eps v)M(1-\eps v)$. 
  Indeed, the transformation preserves the primary part $P$ and takes $D$ to $D+vP-Pv$. Thus 
  \begin{align*}
  \Cj{P}D&\mapsto \Cj{P}D+\Cj{P}vP-\Cj{P}Pv,\\
  D\Cj{P}&\mapsto D\Cj{P}+vP\Cj{P}-Pv\Cj{P},\\
  D\Cj{D}&\mapsto D\Cj{D}+D\Cj{P}\Cj{v}+vP\Cj{D}
        -(D\Cj{v}\Cj{P}+Pv\Cj{D})+(vP-Pv)\Cj{(vP-Pv)}\\
        &=D\Cj{D}+D\Cj{P}\Cj{v}+vP\Cj{D}
        -(\Cj{P}D\Cj{v}+v\Cj{D}P)+(vP-Pv)\Cj{(vP-Pv)}.
  \end{align*}
  Since $c^2|\Cj{P}P,\Cj{P}vP,(vP-Pv)\Cj{(vP-Pv)}$ and $cg|c^2,D\Cj{P},\Cj{P}D$, it follows that the transformation preserves $c$, $g$, the condition $cg|D\Cj{D}$, and thus the assumptions of the theorem.  
  
  Let us show that $\tau(\Cj{P}D+\Cj{P}vP-\Cj{P}Pv)=2\tau(P).$ The condition $\tau(P_1)=\tau(P)$ implies that $Q=(t-p)Q_1$ for some $Q_1\in\H[t]$.
  Applying Lemma~\ref{lem:a-b-c-lemma} twice, we conclude 
  that $N$ does not divide $\Cj{Q_1} (t-\Cj{p})v(t-p)Q_1$ because 
  $N\not|Q=(t-p)Q_1$ and $vp-pv\ne 0$. Thus
\begin{align*}
\tau(\Cj{P}D)&>2\tau(P),\\
\tau(P\Cj{P})&=1+\tau(P_1\Cj{P_1})>2\tau(P_1)=2\tau(P),\\
\tau(\Cj{P}vP)&=2\tau(P)+\tau(\Cj{Q_1} (t-\Cj{p})v(t-p)Q_1)=2\tau(P), 
        \text{ and hence}\\
\tau(\Cj{P}D+\Cj{P}vP-\Cj{P}Pv) &=2\tau(P)
\le\tau(D\Cj{P}+vP\Cj{P}-Pv\Cj{P}). 
\end{align*}

 By Case~1, we get $(1+\eps v)M(1-\eps v)=(t-p)M_1'$ for some $M_1'$ satisfying the assumptions of the theorem. Then $M=(t-p-\eps(pv-vp))(1-\eps v)M_1'(1+\eps v)=:(t-p-\eps(pv-vp))M_1$, where $M_1$ also satisfies the assumptions of the theorem, as required.

[The `only if' part] 
Induction on the degree of $P:=cQ$. It suffices to prove that
if a bounded monic reduced motion polynomial $M_1$ satisfies the condition $c_1g_1|\nu(D_1)$ and $t-p-\eps q$ is a bounded linear motion polynomial, then $M=M_1(t-p-\eps q)$ satisfies the condition $cg|\nu(D)$, once $M$ is also reduced. Here $g_1,c_1,D_1,Q_1$ for $M_1$ are defined analogously to $g,c,D,Q$. Use notation $N:=\nu(t-p)$ as above.

We may assume that $q=0$ by performing the transformation $X\mapsto (1+\eps v)X(1-\eps v)$, with $v=(p^*-p)^{-1}q$, of all the involved motion polynomials $X$. (A geometric explanation: The linear motion polynomial $t-p-\eps q$ describes a rotation. Via a suitable change of coordinates, we may assume that the rotation axis contains the origin whence $q=0$.)

Case 1: $Q_1$ is not right-divisible by $t-p^*$. Then $c=c_1$,
\begin{align*}
g_1&=\mathrm{realgcd}(c_1,Q_1^*D_1,D_1Q_1^*),\\
g&=\mathrm{realgcd}(c_1,(t-p^*)Q_1^*D_1(t-p),ND_1Q_1^*)\,|\,g_1N^2.
\end{align*}
Let us prove that $\tau(g)\le \tau(g_1)+1$; then
$cg\,|\,c_1g_1N\,|\,N\nu(D_1)=\nu(D)$ as required.

Assume that 
$N\,|\,c_1$; otherwise $\tau(g)=0$ and we are done.
Then  $N\,\not|\,D_1(t-p)$;
otherwise $N\,|\,M=c_1Q_1(t-p)+\eps D_1(t-p)$ and $M$ is not reduced.
In particular, $N\,\not|\,Q_1,D_1$. 
By Lemma~\ref{lem:g-a-b-lemma}, we get $Q_1=LQ_2$, $D_1=LD_2$, and $N\,\not|\,Q_2^*D_2$ for some $L,Q_2,D_2\in\H[t]$.

The desired inequality $\tau(g)\le \tau(g_1)+1$ is going to follow from $N^2\,\not|\,(t-p^*)Q_2^*D_2(t-p)$. If $N\,\not|\,(t-p^*)Q_2^*D_2$ or $N\,\not|\,Q_2^*D_2(t-p)$, then we are done. Thus assume $N\,|\,(t-p^*)Q_2^*D_2$ and $N\,|\,Q_2^*D_2(t-p)$.
Then $N\,|\,\nu(Q_2^*D_2)$, hence 
$N\,|\,\QNorm{Q_2}$ or $N\,|\,\QNorm{D_2}$.
Assume that $N\,|\,\nu(D_2)$;
the case $N\,|\,\nu(Q_2)$ is analogous. Then by Lemma~\ref{lem:a-b-c-lemma} we have $N\,|\,Q_2^*D_2$ or $N\,|\,D_2(t-p)$. The former possibility contradicts the choice of $Q_2$ and $D_2$, and the latter one contradicts $N\,\not|\,D_1(t-p)=LD_2(t-p)$.
This proves that $\tau(g)\le \tau(g_1)+1$ and completes Case~1.

Case 2: $Q_1$ is right-divisible by $t-p^*$. Then $c=Nc_1$, $Q_1=Q(t-p^*)$,
\begin{align*}
g_1&=\mathrm{realgcd}(c_1,(t-p)Q^*D_1,D_1(t-p)Q^*),\\
g&=\mathrm{realgcd}(Nc_1,Q^*D_1(t-p),D_1(t-p)Q^*)\,|\,g_1N.
\end{align*}
Here $N\,\not|\,D_1(t-p)\Cj{Q}$; otherwise $M$ or $Q_1$ is not reduced by Lemma~\ref{lem:a-b-c-lemma}.
Hence $N\,\not|\,g$ and $cg\,|\,c_1g_1N\,|\,N\nu(D_1)=\nu(D)$ as required. 
\end{proof}

\begin{algorithm}
  \caption{\texttt{Factor4}: Recursive factorization algorithm}
  \label{alg:recursive}
  \begin{algorithmic}[1]
    \Require 
    A bounded monic reduced motion polynomial $M=P+\eps D \in \D\H[t]$ such that $\mygcd(\mrpf(P)^2, \Cj{P}D, D\Cj{P})|\Cj{D}D$.

    \Ensure A list $[L_1,\ldots,L_n]$ of monic linear motion polynomials such that
    $M=L_1\cdots L_n$.
\If{$\deg \mrpf(P) = 0$}
      \State \Return $[\mathtt{GFactor}(M)]$
    \Else
      \State $N \adef$ any quadratic real factor of $\mrpf(P)$ \State $\tau(X):=$ the maximal $\tau$ such that $N^\tau|X$.  
      \If  {$\tau(D\Cj{P}) < \tau(\Cj{P}D)$} 
      \State \Return $[\texttt{Factor4}(\Cj{M})]^\ast$ (the list with the order reversed and elements conjugated)
      \Else
      \State $p \adef$ any quaternion such that $(t-p)(t-\Cj{p})=N$ and $D$ is left-divisible by $t-p$ \State $P_1 \adef (t-p)^{-1}P$, $D_1 \adef (t-p)^{-1}D$
\If {$\tau(P_1)<\tau(P)$ or $\tau(\Cj{P}D)\le 2\tau(P)$}
      \State $M_1\adef P_1+\eps D_1$  
      \State \Return $[t-p,\texttt{Factor4}(M_1)]$
      \Else \State $v \adef$ any vectorial quaternion that does not commute with $p$
      \State $q \adef pv-vp$
      \State $M_1\adef P_1+\eps (D_1+qP/N)$  
      \State \Return $[t-p-\eps q,\texttt{Factor4}(M_1)]$
      \EndIf
      \EndIf
    \EndIf

  \end{algorithmic}
\end{algorithm}

The `if' part in this proof gives rise to recursive
Algorithm~\ref{alg:recursive} and also constitutes the proof of its correctness. The resulting algorithm is apparently very different from Algorithms~\ref{alg:mgfactorization} and~\ref{alg:singlepattern} from the previous section. However, they still can output the same linear factors in many cases. For instance, for the example in Section~\ref{sec:example} and appropriate choices of $N$ and $q$, Algorithm~\ref{alg:recursive} ends up with the same factorization.

}

\begin{cor}\label{cor:reverse-final}
  Consider a bounded, monic, reduced motion polynomial $M = c Q + \eps D \in
  \DH[t]$ with $c \in \R[t]$ and $\mrpf(Q) = 1$. Set $g \coloneqq
  \mygcd(c,\Cj{Q}D,D\Cj{Q}) \in \R[t]$ and $g' \coloneqq
  \frac{cg}{\mygcd(cg,\nu{(D)})} \in \R[t]$. Then $Mg'$ admits a factorization
  into monic linear factors.
\end{cor}
\begin{proof}
  The proof is by induction on $\deg g'$, which is an even number because $M$ is bounded. 
  
  The base $\deg g'=0$ is Theorem~\ref{th:bounded}. 
   
  To perform the induction step, take a bounded monic reduced motion polynomial $M_1$ such that $\deg g_1'\ne 0$ Here $g_1',g_1,c_1,D_1,Q_1$ for $M_1$ are defined analogously to $g',g,c,D,Q$. Pick a monic real quadratic irreducible factor $N$ of $g'_1$. Assume without loss of generality that $\tau(\Cj{Q_1}D_1) \leq \tau(D_1\Cj{Q_1})$, because $M_1g_1'$ admits a factorization if and only if $\Cj{M_1}g_1'$ does, and the conjugation of $M_1$ interchanges $\Cj{Q_1}D_1$ and $D_1\Cj{Q_1}$. Divide the ratio $\Cj{Q_1}D_1/\mrpf(\Cj{Q_1}D_1)$ by $N$ with a remainder $qt+r$. Since $M_1$ is a motion polynomial, it follows that $\Cj{Q_1}D_1+\Cj{D_1}Q_1=0$, hence $q$ and $r$ are vectorial quaternions not vanishing simultaneously. There are infinitely many quaternions $p$ such that $\nu(t-p)=N$ and $p(pq+r)\ne (pq+r)p$. Indeed, consider the vectorial quaternion $r':=r+q(p+\Cj{p})/2$ determined by $q$, $r$, and $\nu(t-p)=N$. 
  Take the vectorial part of $p$ to be orthogonal to $q$ if $r'=0$ or $q$ is non-orthogonal to $r'$, orthogonal to $r'$ if $q=0$, and orthogonal to neither $q$ nor $r'$ if $q\ne 0$ is orthogonal to $r'\ne 0$. Then $p(pq+r)-(pq+r)p$ is not orthogonal to $r'$ (for $q,r'\ne 0$) and hence nonzero. Among the infinitely many $p$, pick one such that $t-\Cj{p}$ is not a right factor of $Q_1$ and $D_1$. 

  Take $M:=M_1(t-p)$. We have $\mrpf(\Cj{Q}D)=\mrpf(\Cj{Q_1}D_1)$
  because 
  \begin{align*}
  \frac{\Cj{Q}D}{\mrpf(\Cj{Q_1}D_1)}
  &=(t-\Cj{p})\frac{\Cj{Q_1}D_1}{\mrpf(\Cj{Q_1}D_1)}(t-p)=(t-\Cj{p})(qt+r)(t-p)\\
  &=(t-\Cj{p})((t-p)q+pq+r)(t-p)=(t-\Cj{p})(pq+r)(t-p)\\
  &=(t-\Cj{p})(p(pq+r)-(pq+r)p)\ne 0\mod N.
  \end{align*}
  Hence $c=c_1$, $g=g_1$, and $g'=g_1'/N$. By the inductive hypothesis, $Mg'$ admits a factorization into monic linear factors. Hence $M_1g_1'=Mg'(t-\Cj{p})$ does.
\end{proof}

A natural open question is if the reciprocal assertion holds in the sense that the polynomial~$g'$ divides any real polynomial $S$ such that $MS$ admits a factorization into monic linear factors. It is also interesting to generalize Corollary~\ref{cor:reverse-final} to non-reduced motion polynomials $M$.

\section{Conclusion}
\label{sec:conclusion}

Motion polynomials can describe rational motions but a rational motion
determines the underlying motion polynomial only up to real polynomial factors
\cite{li16}. Therefore, it is natural to consider reduced motion polynomials,
which admit a factorization into linear factors. In this paper, we first gave a
necessary condition (Proposition~\ref{prop:final-characterization}) for a reduced
monic motion polynomial to admit such a decomposition. This condition is also
sufficient (Proposition~\ref{prop:reverse-final}) and can easily be verified.
Furthermore, we provided a constructive algorithm for finding a factorization
when the necessary condition is fulfilled. This also reveals an improved method
to find a real polynomial such that the product with the given motion polynomial
admits a factorization (Corollary~\ref{cor:reverse-final}).

Factorizations into linear factors lie at the core of some mechanism
constructions, for example, the universal construction of Kempe linkages of
\cite{li17} for rational trajectories where the numbers of links and joints are
bounded by a linear function in the curve degree. In that paper, the authors were
interested in linkages with a prescribed rational trajectory and could construct
``tame'' motion polynomials (essentially having $g = 1$ in
Propositions~\ref{prop:final-characterization} and~\ref{prop:reverse-final}). Using the
result herein it is possible to construct linkages with even fewer links and
joints to follow a rational space motion.

The procedure, even if spread about several sub-algorithms, is basically rather
simple, only involving dual quaternion algebra, $\mygcd$-computations and polynomial division. It is symbolic but extensions to numeric input data should
be possible using algorithms for approximate $\mygcd$-computations \cite{qrgcd04,kaltofen06,pan2001gcd}.

Having characterized factorizability in the motion polynomial case, a natural
next step is the extension of our results to \emph{spinor polynomials}. These
generalize the concept of motion polynomial to Clifford algebras and can be used
to describe rational motion in groups different from~$\SE$. The behaviour of
spinor polynomials is similar to motion polynomials: Factorizations exist
generically but not generally. In conformal three-dimensional geometric algebra,
suitable polynomial multiples of spinor polynomials always admit factorizations
\cite{li23}. Identifying factorizability conditions for spinor polynomials along the lines of
our Proposition~\ref{prop:final-characterization} is an interesting open problem. 

Polynomials over the dual quaternions can also be viewed as quaternion polynomials in two variables $t$ and $\varepsilon$ considered up to adding a multiple of $\varepsilon^2$: $\DH[t]=\H[t,\varepsilon]/(\varepsilon^2)$. So, yet another direction is to extend our results to quaternion polynomials in two variables.

\section*{Acknowledgments}
Zijia Li is supported by the National Key R$\&$D Program of China (2023YFA1009401) and partially supported by the Strategic Priority Research Program of the Chinese Academy of Sciences 0640000 \& XDB0640200 and the Guangdong Basic and Applied Basic Research Foundation under Grant 2024A1515010506. Mikhail Skopenkov is supported in part by the KAUST baseline. Daniel Scharler was supported by the Austrian Science Fund (FWF) P~33397-N (Rotor
Polynomials: Algebra and Geometry of Conformal Motions). The first two authors acknowledge the support of the Institut Henri Poincar\'e (UAR 839 CNRS-Sorbonne Universit\'e) and LabEx CARMIN (ANR-10-LABX-59-01). 
 
\bibliographystyle{plainnat}

\end{document}